\documentclass[hidelinks, 10pt]{amsart}

\usepackage[margin=3cm]{geometry}

\usepackage{amssymb}

\usepackage{cite}
\usepackage{hyperref}
\usepackage[all]{xy}
\usepackage{enumerate}
\usepackage{bbm}
\usepackage{tikz-cd}
\usepackage{extpfeil}
\usepackage{fancyhdr}

\newcommand{\bC}{{\mathbb C}}

\newcommand{\bG}{{\mathbb G}}

\newcommand{\bK}{{\mathbb K}}
\newcommand{\bL}{{\mathbb L}}
\newcommand{\bM}{{\mathbb M}}
\newcommand{\bN}{{\mathbb N}}

\newcommand{\bQ}{{\mathbb Q}}

\newcommand{\bZ}{{\mathbb Z}}

\newcommand{\cA}{{\mathcal A}}

\newcommand{\cF}{{\mathcal F}}

\newcommand{\cH}{{\mathcal H}}

\newcommand{\cO}{{\mathcal O}}

\newcommand{\cR}{{\mathcal R}}

\newcommand{\cU}{{\mathcal U}}

\newcommand{\cX}{{\mathcal X}}
\newcommand{\cY}{{\mathcal Y}}

\newcommand{\fD}{{\mathfrak D}}

\newcommand{\fX}{{\mathfrak X}}

\newcommand{\oK}{\overline{K}}

\newcommand{\ox}{\overline{x}}
\newcommand{\oy}{\overline{y}}
\newcommand{\obQ}{\overline{\mathbb{Q}}}

\newcommand{\tV}{\widetilde{V}}

\newcommand{\et}{\text{et}}

\DeclareMathOperator{\gr}{{gr}}

\DeclareMathOperator{\rk}{{rk}}

\DeclareMathOperator{\End}{{End}}

\DeclareMathOperator{\Ind}{{Ind}}

\DeclareMathOperator{\id}{{id}}

\DeclareMathOperator{\Spec}{{Spec}}

\DeclareMathOperator{\coker}{coker}

\DeclareMathOperator{\Res}{{Res}}

\DeclareMathOperator{\Gal}{{Gal}}

\newcommand{\alg}{\mathrm{alg}}
\newcommand{\ip}{\frac{1}{p}}

\newcommand{\dR}{\mathrm{dR}}
\newcommand{\HT}{\mathrm{HT}}

\DeclareMathOperator{\Spa}{{Spa}}
\newcommand{\fin}{\mathrm{fin}}
\newcommand{\cycl}{\mathrm{cycl}}
\newcommand{\Higgs}{\mathrm{Higgs}}

\newcommand{\proj}{\mathrm{proj}}
\newcommand{\Mat}{\mathrm{Mat}}

\newcommand{\RH}{\mathrm{RH}}
\newcommand{\Hig}{\mathrm{H}}

\newcommand{\Id}{\mathrm{Id}}
\newcommand{\Fil}{\mathrm{Fil}}

\newcommand{\hkinf}{\widehat{K}_{\infty}}
\newcommand{\Kinf}{K_{\infty}}
\newcommand{\oX}{\overline{X}}
\newcommand{\ofX}{\overline{\fX}}
\newcommand{\os}{\overline{s}}
\newcommand{\oF}{\overline{F}}
\newcommand{\oL}{\overline{L}}
\newcommand{\oY}{\overline{Y}}
\newcommand{\opi}{\overline{\pi}}
\newcommand{\obZ}{\overline{\bZ}}
\newcommand{\ubQp}{\underline{\bQ_p}}
\newcommand{\an}{\mathrm{an}}
\newcommand{\proet}{\mathrm{proet}}
\newcommand{\pro}{\mathrm{pro}}

\makeatletter

\newcommand{\Rmnum}[1]{\expandafter\@slowromancap\romannumeral #1@}
\makeatother

\newtheorem{manualtheoreminner}{Conjecture}
\newenvironment{manualconj}[1]{%
  \renewcommand\themanualtheoreminner{#1}%
  \manualtheoreminner
}{\endmanualtheoreminner}

\newtheorem{conj}{Conjecture}

\newtheorem{pr}{Proposition}[section]

\newtheorem{thm}[pr]{Theorem}

\newtheorem{Th}{Theorem}

\newtheorem{lm}[pr]{Lemma}

\newtheorem{cor}[pr]{Corollary}

\theoremstyle{definition}
\newtheorem{rem}[pr]{Remark}

\theoremstyle{definition}
\newtheorem{example}[pr]{Example}
\newtheorem{question}[pr]{Question}
\newtheorem{definition}[pr]{Definition}

\numberwithin{equation}{section}

\begin{document}

\title[Geometrically irreducible $p$-adic local systems are de Rham up to a twist]
{Geometrically irreducible $p$-adic local systems are de Rham up to a twist}

\author[]{Alexander Petrov}
\address{Harvard University, USA}
\email{apetrov@math.harvard.edu}

\begin{abstract} 
We prove that any geometrically irreducible $\obQ_p$-local system on a smooth algebraic variety over a $p$-adic field $K$ becomes de Rham after a twist by a character of the Galois group of $K$. In particular, for any geometrically irreducible $\obQ_p$-local system on a smooth variety over a number field the associated projective representation of the fundamental group automatically satisfies the assumptions of the relative Fontaine-Mazur conjecture. The proof uses $p$-adic Simpson and Riemann-Hilbert correspondences of Diao-Lan-Liu-Zhu and the Sen operator on the decompletions of those developed by Shimizu. Along the way, we observe that a $p$-adic local system on a smooth geometrically connected algebraic variety over $K$ is Hodge-Tate if its stalk at one closed point is a Hodge-Tate Galois representation. Moreover, we prove a version of the main theorem for local systems with arbitrary geometric monodromy, which allows us to conclude that the Galois action on the pro-algebraic completion of $\pi_1^{\et}$ is de Rham.

 \end{abstract}

\maketitle

\section{Introduction}

Let $X$ be a smooth algebraic variety over a complete discretely valued field $K$ of characteristic $0$ with a perfect residue field of characteristic $p>0$.  Inside the category of $\bQ_p$-etale local systems on $X$ there is the subcategory of de Rham local systems \cite{scholze-rig}, \cite{lz}. To any de Rham local system one can canonically attach a vector bundle of the same rank with integrable connection and filtration satisfying Griffiths transversality. The primary examples of de Rham local systems are given by relative cohomology $R^n\pi_*\bQ_p$ of smooth proper morphisms $\pi:Y\to X$. Our main result is the following

\begin{Th}[Corollary \ref{main th as cor}]\label{main theorem}
Suppose that $\bL$ is a $\obQ_p$-local system on $X$ such that the restriction of $\bL$ to $X_{\oK}$ has no non-scalar endomorphisms: $\End_{X_{\oK}}(\bL)=\obQ_p\cdot\Id$. Then there exists a character of the Galois group $\chi:G_{K}\to \obQ_p^{\times}$ such that $\bL\otimes\chi$ is de Rham .
\end{Th}

This fact has been conjectured by Daniel Litt. More generally, our proof works for $X$ a smooth rigid analytic variety that can be embedded into a smooth proper $\oX$ such that $\oX\setminus X$ is a normal crossings divisor, and assuming additionally that $\bL$ descends to an $\cO_{E}$-local system for a finite extension $E\supset \bQ_p$. Theorem \ref{main theorem} has the following implications on our expectations of which conditions should be sufficient to guarantee that a local system comes from geometry. We say that a $\obQ_p$-etale local system on a scheme $X$ {\it comes from geometry} if there exists a dense open subscheme $U\subset X$ and a morphism $\pi:Y\to U$ such that $\bL|_U$ is a subquotient of $R^i\pi_*\obQ_p(j)$ for some $i,j$. Recall the relative form of the Fontaine-Mazur conjecture \cite{fm}, \cite{lz}:

\begin{conj}\label{fm original}
Suppose that an irreducible $\obQ_p$-local system $\bL$ on a smooth variety $X$ over a number field $F$ satisfies the conditions 

1)There is a finite set $S$ of places of $F$ and a smooth model $\fX$ of $X$ over $\cO_{F,S}$ such that $\bL$ extends to a local system on $\fX$.

2)For every place of $F$ above $p$ the restriction of $\bL$ to $X_{F_v}$ is a de Rham local system.

Then $\bL$ comes from geometry.
\end{conj}

The above theorem implies that for a geometrically irreducible local system the second condition is automatically satisfied for the corresponding projective representation of the arithmetic fundamental group. First condition is also easily seen to be automatic (as proven in \cite{litt}, we recall the proof in Proposition \ref{automatically unramified}) yielding:

\begin{Th}\label{fm intro}
Let $\bL$ be a $\obQ_p$-local system on a smooth variety $X$ over a number field $F$ with the corresponding representation $\rho:\pi_1^{\et}(X,\ox)\to GL_d(\obQ_p)$. If $\rho|_{\pi_1^{\et}(X_{\oF},\ox)}$ is irreducible then

1)$\rho$ extends to $\pi_1^{\et}(\fX,\ox)$ where $\fX$ is a smooth scheme over $\cO_{F,S}$ for some finite set of places $S$ with $\fX\times_{\cO_{F,S}} F=X$.

2)For every place $v$ of $F$ above $p$ the projective representation $\rho_{\proj}|_{F_v}:\pi_1^{\et}(X_{F_v},\ox)\to GL_d(\obQ_p)\to PGL_{d}(\obQ_p)$ is de Rham.
\end{Th}

The conclusion (2) of Theorem \ref{fm intro} is weaker than saying that $\bL$ can be made de Rham by twisting it by a character of $G_F$ as \cite[Example 6.8]{conradnote} shows. This issue does not occur when working with local systems over local fields by  a result of Patrikis \cite[Corollary 3.2.13]{patrikis}. This allows us to prove Theorem \ref{main theorem} in the stated form. Nevertheless, Moritz Kerz observed (Lemma \ref{square trick}) that Theorem \ref{fm intro} yields the following corollary of the relative Fontaine-Mazur conjecture:

\begin{manualconj}{1 bis}
Let $\bL$ be a $\obQ_p$-local system on a smooth variety $X$ over a number field $F$ such that the local system $\bL|_{X_{\oF}}$ on $X_{\oF}$ is irreducible. Then $\bL|_{X_{\oF}}$ comes from geometry.
\end{manualconj}

Alexander Beilinson observed that the above results can be generalized to the case of not necessarily irreducible local systems. Using the same techniques as in the proof of Theorem \ref{main theorem}, we show

\begin{Th}[Theorem \ref{main reducible}]
For any $\obQ_p$-local system $\bL$ on a smooth variety $X$ over $K$ there exists a de Rham local system $\bM$ such that there is an inclusion $\bL|_{X_{\oK}}\subset \bM|_{X_{\oK}}$ of local systems on $X_{\oK}$. 
\end{Th}

This result can be restated in terms of the Galois action on the pro-algebraic completion of the fundamental group of $X_{\oK}$, generalizing the results of \cite{vologodsky}, \cite{olsson} regarding the pro-unipotent completion.

\begin{cor}[Proposition \ref{galois on proalg pi1}]
Any finite-dimensional subrepresentation $V\subset \obQ_p[\pi_1^{\pro-\alg}(X_{\oK},\ox)]$ of the Galois group $G_K$ is de Rham.
\end{cor}

We will now sketch the proof of Theorem \ref{main theorem} in the case where $X$ is proper assuming also that $\bL$ is a $\bQ_p$-local system with $\End_{X_{\oK}}\bL=\bQ_p\cdot\Id$ with the conclusion of the theorem weakened to $\bL\otimes_{\bQ_p}\chi$ being Hodge-Tate. The main tool used in the proof is a decompleted version of the $p$-adic Simpson correspondence as developed by \cite{lz} and \cite{kht}, summarized in the following. Put $\Kinf=K(\mu_{p^n}|n\in\bN)$.

\begin{pr}\label{intro simpson summary}
For a smooth proper algebraic variety $X$ over $K$ to any $\bQ_p$-local system $\bL$ on $X$ one can associate a Higgs bundle $(\Hig(\bL),\theta)$ on $X_{\Kinf}$ equipped with an $\cO_{X_{\Kinf}}$-linear endomorphism $\phi:\Hig(\bL)\to \Hig(\bL)$ such that $\theta\circ\phi=(\phi-1)\circ\theta$. This construction satisfies the following properties:

$(\mathrm{i})$ $\Hig$ is a tensor functor compatible with duality.

$(\mathrm{ii})$ We have $\dim_{\bQ_p}H^i_{\et}(X_{\oK},\bL)=\dim_{\Kinf}H^i_{\Higgs}(X_{\Kinf},\Hig(\bL))$.

$(\mathrm{iii})$ The characteristic polynomial $\det(T\cdot\Id_{\Hig(\bL)}-\phi)$ of $\phi$ has coefficients in $K$. A local system $\bL$ is Hodge-Tate if and only if $\phi$ is a semi-simple operator with eigenvalues in $\bZ$.
\end{pr}

The relation $\theta\circ \phi=(\phi-1)\circ \theta$ implies that generalized  eigenspace decomposition of $\Hig(\bL)$ with respect to the endomorphism $\phi$ gives $(\Hig(\bL),\theta)$ a structure of a system of Hodge bundles in the sense of \cite[Section 4]{simpson}. By Proposition \ref{intro simpson summary}, the condition $\End_{X_{\oK}}\bL=\bQ_p\cdot\Id$ implies that all endomorphisms of the Higgs bundle $(\Hig(\bL),\theta)$ are scalars. In particular, it is indecomposable and therefore eigenvalues of $\phi$ must lie in a subset of the form $a+\bZ\subset\oK$ for some $a\in\oK$. Twisting by an appopriate character we may assume that $a=0$ so $\phi$ has integral eigenvalues and, if $\phi$ was not semi-simple, the Higgs bundle would have a non-zero nilpotent endomorphism. Therefore, the twsited local system is Hodge-Tate.

We then deduce (Theorem \ref{de rham main}) that a twist of $\bL$ is not just Hodge-Tate but is also de Rham by employing the $p$-adic Riemann-Hilbert correspondence of \cite{lz}, \cite{dllz}. Building on the work of \cite{kht}, we obtain (Proposition \ref{main decompleted rh}) a decompleted version $\RH(\bL)$ of the correspondence. The main observation then (Proposition \ref{mock Ddr}) is that for any local system $\bL$ with integral Hodge-Tate weights one can strip off the associated $\Kinf((t))$-linear filtered vector bundle with connection $\RH(\bL)$ to get a $\Kinf$-linear filtered vector bundle with connection whose associated graded is $\Hig(\bL)$. One then utilizes the same idea as above: if $\bL$ was not de Rham, the Sen endomorphism would provide a non-scalar endomorphism of the Higgs bundle $\Hig(\bL)$.

To generalize this proof from $\bQ_p$ local systems to $\obQ_p$-local systems one simply needs to redo the above argument with $\bL$ equipped with an action of a finite extension $E$ of $\bQ_p$. To cover the case of a non-proper $X$ we use the logarithmic versions of $p$-adic Simpson and Riemann-Hilbert correspondences developed in \cite{dllz}.

As a byproduct of the above method, we show the following rigidity statement for the Hodge-Tate property of fibers of a local system on an algebraic variety

\begin{pr}[Proposition \ref{ht locally free}  \& Corollary \ref{ht rigidity main}]
If $\bL$ is a $\bQ_p$-local system on a smooth {geometrically} connected algebraic variety $X$ over $K$, then the Higgs sheaf $D_{\HT}(\bL)$ is locally free. In particular, $\dim_{k(x)} D_{\HT}(\bL_{\ox})$ does not depend on the choice of a closed point $x\in X$ and, if the representation of the Galois group $\Gal{\overline{k(x)}/k(x)}$ is Hodge-Tate for some closed point $x\in X$ then the local system $\bL$ is Hodge-Tate.
\end{pr}

The proof rests on the observation that the Higgs bundles attached to local systems on connected algebraic varieties via $p$-adic Simpson correspondence satisfy a special property (Lemma \ref{vanishing rigidity}): if a section annihilated by the Higgs field vanishes at one point then it is identically zero. This forces the Sen operator to be semi-simple if it is known that it is semi-simple at one point. The statement of this result is similar to Liu-Zhu's rigidity theorem for the de Rham condition \cite[Theorem 1.3]{lz} and Shimizu's constancy of Hodge-Tate weights that was already used above. The proofs, however, have different flavors and are exploiting entirely different effects: the rigidity theorems of \cite{lz} and \cite{kht} are local statements that are proven for arbitrary smooth rigid-analytic bases while our result is global: it uses crucially the comparison between the global sections of a local system and the associated Higgs bundle. It might well be the case that the rigidity is false for local systems on non-partially proper rigid analytic varieties.
 
{\bf Acknowledgements.} I am grateful to Alexander Beilinson, H\'el\`ene Esnault, Mark Kisin, Moritz Kerz, Dmitry Kubrak, Daniel Litt, Koji Shimizu, and Bogdan Zavyalov for several interesting and useful discussions. I am especially thankful to Alexander Beilinson for his generalization of the main result to the non-irreducible case and for the idea of looking at these results from the vantage point of Galois action on the fundamental group. I also want to thank H\'el\`ene Esnault, Mark Kisin, Shizhang Li, and Koji Shimizu for the comments on the draft of this text.

\section{Preliminaries on $p$-adic Riemann-Hilbert correspondence}
\label{section2}

Generalizing Fontaine's notions of Hodge-Tate and de Rham representations of the Galois group of $K$, Brinon \cite{brinon}, Scholze \cite{scholze-rig} and in this particuar form Liu and Zhu \cite{lz} have constructed the following functors for any smooth rigid-analytic variety $X$ over $K$. First, there is a functor $$D_{\dR}:\{\bQ_p-\text{local systems on }X\}\rightarrow \{\parbox{15em}{\centering Filtered vector bundles on $X$ with a Griffiths transverse connection}\}$$ A local system $\bL$ is called {\it de Rham} if the output of this functor $D_{\dR}(\bL)$ has rank equal to $\rk_{\bQ_p}\bL$. Likewise, there is a functor $$D_{\HT}:\{\bQ_p-\text{local systems on }X\}\rightarrow\{\parbox{9em}{\centering Coherent sheaves on $X$ with a Higgs field}\}$$ A {\it Hodge-Tate} local system is one for which $D_{\HT}(\bL)$ is a locally free sheaf of rank $\rk_{\bQ_p}\bL$. For convenience of notation let us pick a geometric point $\ox$ of $X$. Let $\pi_1^{\et}(X,\ox)$ denote the fundamental group of $X$ in the sense of de Jong \cite{de-jong}.

\begin{definition}
An etale $\obQ_p$-local system of rank $d$ on a rigid-analytic variety $X$ is a continuous homomorphism $\rho:\pi_1^{\et}(X,\ox)\to GL_d(\obQ_p)$ that factors through $GL_d(E)$ for some finite extension $E\supset \bQ_p$. We say that an $\obQ_p$-local system admits a $\obZ_p$-lattice if $\rho$ is, moreover, conjugate to a homomorphism into $GL_d(\cO_E)\subset GL_d(\obQ_p)$. This condition is equivalent to $\rho$ factoring through the profinite completion of $\pi_1^{\et}(X,\ox)$.
\end{definition}

If $X$ appears as the analytification of an algebraic variety then any $\obQ_p$-local system on this algebraic variety gives rise to a $\obQ_p$-local system in the above sense and it automatically admits a $\obZ_p$-lattice. The proofs of all our results will require the existence of $\obZ_p$-lattices.

\begin{definition}
For a reductive group $G$ over $\bQ_p$ a continuous representation $\pi_1^{\et}(X,\ox)\to G(\bQ_p)$ is {\it de Rham} (resp. {\it Hodge-Tate}) if for some (equivalently, every) embedding $G\hookrightarrow GL_{n,\bQ_p}$ the rank $n$ $\bQ_p$-local system corresponding to $\pi_1^{\et}(X,\ox)\to G(\bQ_p)\to GL_n(\bQ_p)$ is de Rham (resp. {Hodge-Tate}). For any finite extension $E$ of $\bQ_p$ an $E$-local system $\bL$ on $X$ is called {\it de Rham} (resp. {\it Hodge-Tate}) if the corresponding representation into $(\Res^E_{\bQ_p}GL_n)(\bQ_p)$ is de Rham (resp. {Hodge-Tate}).  We call a $\obQ_p$-local system {\it de Rham} (resp. {\it Hodge-Tate}) if any (equivalently, every) of its descents to an $E$-local system for some finite extension $E/\bQ_p$ is de Rham (resp. {Hodge-Tate}).
\end{definition}

The property of being de Rham or Hodge-Tate indeed does not depend on the choice of a faithful representation $G\hookrightarrow GL_{n,\bQ_p}$ because any other representation of the reductive group $G$ can be embedded into a representation of the form $V^{\otimes a}\otimes V^{\vee\otimes b}$ where $V$ denotes the tautological representation of $GL_{n,\bQ_p}$.

\begin{example}\label{projective rep example}
According to this definition, for a $\obQ_p$-local system $\bL$ on $X$ the corresponding projective representation $\rho:\pi_1^{\et}(X,\ox)\to PGL_n(\obQ_p)$ is de Rham (resp. Hodge-Tate) if and only if its composition with the adjoint representation $PGL_{n,\obQ_p}\to GL_{n^2,\obQ_p}$ has the corresponding property. In terms of the local system $\bL$ this is equivalent to saying that the local system $\bL\otimes_{\obQ_p}\bL^{\vee}$ is de Rham (resp. Hodge-Tate).
\end{example}

We summarize in this section the setup and the properties of $p$-adic Simpson and Riemann-Hilbert correspondences of \cite{lz} and \cite{dllz}. From now on suppose that $X$ can be embedded into a smooth proper rigid-analytic variety $\oX$ such that the complement $D=\oX\setminus X$ is a normal crossings divisor. 

Let $K_n=K(\mu_{p^n})$ for $n\geq 0$ denote the $p$-cyclotomic extensions of $K$ and let $\Kinf=\bigcup\limits_{n\geq 1}K_n$ be their union with  $\hkinf$ denoting its completion. The Galois group $\Gal(\Kinf/K)$ is an open subgroup of $\bZ_p^{\times}$ which we denote by $\Gamma$. We will denote by $\bQ_p(1)=(\lim\limits_{n}\mu_{p^n}(\oK))[\ip]$ the cyclotomic character of $G_K$, sometimes viewed as a character of $\Gamma$. Denote also by $H_K$ the subgroup $\Gal(\oK/\Kinf)\subset G_K$. Evaluating Fontaine's construction of $B_{\dR}$ on the perfectoid field $\hkinf$ we get rings $L_{\dR}=B_{\dR}(\hkinf),L_{\dR}^+=B_{\dR}^+(\hkinf)$ equipped with the filtrations such that $L_{\dR}^+$ is complete. For the convenience of notation we fix a compatible system of roots of unity $\zeta_p,\zeta_{p^2},\ldots\in \hkinf$ and denote the corresponding element $\log[(1,\zeta_p,\zeta_{p^2},\ldots)]\in L_{\dR}^+$ by $t$. Denote also by $C$ the completed algebraic closure $\widehat{\overline{K}}$ of $K$.

Consider the ringed space $\overline{\cX}$ defined as $(\oX_{\hkinf},(\cO_{\oX}\widehat{\otimes_K}L_{\dR}^+)[t^{-1}])$. It is equipped with the sheaf of logarithmic differentials $\Omega^{1,\log}_{\overline{\cX}/L_{\dR}}=\Omega^{1,\log}_{\oX/K}\widehat{\otimes}_K L_{\dR}$ where $\Omega^{1,\log}_{\oX/K}$ is the sheaf of $1$-forms that are regular on $X$ and have at worst logarithmic poles along $D$. See \cite[Definitions 3.5, 3.6]{lz} and \cite[Subsection 3.1]{dllz} for complete definitions of these objects and for the notion of a filtered vector bundle with connection on $\overline{\cX}$.

Here is the excerpt from \cite{dllz} that we will need:

\begin{thm}\label{liu zhu rh}
There exists an exact functor $\bL\mapsto \cR\cH^{\log}(\bL)$ from $\bQ_p$-local systems on $X$ to $\Gamma$-equivariant filtered vector bundles on $\overline{\cX}$ with a logarithmic integrable connection that satisfies Griffiths transversality such that 


(i) Denoting by $\cH^{\log}(\bL)$ the 0th graded piece $\gr^0\cR\cH^{\log}(\bL)$ we have canonical $\Gamma$-equivariant isomorphisms $\gr^i\cR\cH^{\log}(\bL)\simeq \cH^{\log}(\bL)(i)$. The connection on $\cR\cH^{\log}(\bL)$ induces a $\Gamma$-equivariant Higgs field $$\theta:\cH^{\log}(\bL)\to\cH^{\log}(\bL)(-1)\otimes_{\cO_{\oX_{\hkinf}}}\Omega^{1,\log}_{\oX_{\hkinf}/\hkinf}$$

(ii) The restrictions $\bL\mapsto \cH^{\log}(\bL)|_{X_{\hkinf}}$ and $\bL\mapsto \cR\cH^{\log}(\bL)|_{\cX}$ are monoidal functors compatible with duality. 

(iii) If $\bL$ admits a $\bZ_p$-lattice then there is a $H_K$-equivariant isomorphism (with respect to the trivial action on $\End_{\oX_{\hkinf},\Higgs}\cH^{\log}(\bL)$) of $C$-algebras $$\End_{\oX_{\hkinf},\Higgs}\cH^{\log}(\bL)\otimes_{\hkinf}C\simeq \End_{X_{\oK}}\bL|_{X_{\oK}}\otimes_{\bQ_p}C.$$

(iv) We have $$\cR\cH^{\log}(\bL)^{\Gamma}|_X\simeq D_{\dR}(\bL) \qquad (\bigoplus\limits_{i\in\bZ}\cH^{\log}(\bL)(i))^{\Gamma}|_X\simeq D_{\HT}(\bL)$$

(v) For any classical point $x\in X$, choosing a geometric point $\ox:\Spa(\oK,\cO_{\oK})\to X$ above it, we get a $\Gal(\overline{k(x)}/k(x)(\mu_{p^{\infty}}))$-equivariant isomorphism (with respect to trivial action on $\cH^{\log}(\bL)_x$) $$\cH^{\log}(\bL)_x\otimes_{\widehat{k(x)(\mu_{p^{\infty}})}}C\simeq \bL_{\ox}\otimes_{\bQ_p}C$$

\end{thm}

\begin{proof}
Part (i) is \cite[Theorems 3.2.3 and 3.2.4]{dllz}, part (ii) is \cite[Theorem 2.1 (iv) and Theorem 3.8 (i)]{lz}, and part (v) is \cite[Theorem 2.1 (ii)]{lz}. Part (iv) is the definition of functors $D_{\dR}$ \cite[Section 3.2]{lz} and $D_{\HT}(\bL)$ \cite[Section 5]{kht} together with the fact that $\cR\cH^{\log}(\bL)|_X\simeq \cR\cH(\bL)$ and $\cH^{\log}(\bL)|_X\simeq \cH(\bL)$. 

The statement (iii) would be immediate from the Hodge-Tate comparison \cite[Theorem 3.2.4(ii)]{dllz} applied to $\bL\otimes_{\bQ_p}\bL^{\vee}$, if $\cH^{\log}$ was monoidal. Unfortunately, in general only a weaker statement hold:

\begin{pr}\label{monoidality rh}
The functor $\bL\mapsto \cH^{\log}(\bL)$ is lax monoidal: for any $\bL_1,\bL_2$ there exists a canonical map $\cH^{\log}(\bL_1)\otimes_{\cO_{\oX_{\hkinf}}}\cH^{\log}(\bL_2)\to\cH^{\log}(\bL_1\otimes_{\bQ_p}\bL_2)$. This map is an isomorphism if both $\bL_1$ and $\bL_2$ have unipotent (as opposed to merely quasi-unipotent) monodromy along $D$.
\end{pr}

\begin{proof}
The lax monoidality is immediate from the definition of $\cH^{\log}$ because the period sheaf $\cO\bC_{\log,\oX}$ is a sheaf of rings. The second statement is \cite[Theorem 3.2.12(1)]{dllz}
\end{proof}

This issue is not specific to the $p$-adic situation and comes down to the fact that, when defining $\cR\cH^{\log}$ and $\cH^{\log}$ via Deligne's canoncial extension, we choose the residues of the extension to lie in $[0,1)$, but the residues of the tensor product $\cR\cH^{\log}(\bL_1)\otimes_{\cO_{\overline{\cX}}}\cR\cH^{\log}(\bL_2)$ need not belong to $[0,1)$.

Nevertheless, we will prove that for any $\bZ_p$-local system $\bL$ on $X$ there is a canonical isomorphism $$H^0_{\Higgs}(\oX_{\hkinf},\cH^{\log}(\bL\otimes_{\bZ_p}\bL^{\vee}))\simeq H^0_{\Higgs}(\oX_{\hkinf},\cH^{\log}(\bL)\otimes_{\oX_{\hkinf}}\cH^{\log}(\bL)^{\vee})$$ which will yield the result by \cite[Theorem 3.2.4(ii)]{dllz}. We have a canonical map of Higgs bundles $\cH^{\log}(\bL\otimes_{\bZ_p}\bL^{\vee})\otimes_{\cO_{\oX_{\hkinf}}}\cH^{\log}(\bL)\to \cH^{\log}(\bL)$ that induces the map \begin{equation}\label{laxmonoidal}\alpha_{\bL}:\cH^{\log}(\bL\otimes_{\bZ_p}\bL^{\vee})\to\cH^{\log}(\bL)\otimes_{\cO_{X_{\hkinf}}}\cH^{\log}(\bL)^{\vee}\end{equation}

which is an isomorphism when restricted to $X_{\hkinf}\subset\oX_{\hkinf}$ and, in particular, is injective. We will prove that $\alpha_{\bL}$ induces an isomorphism on $H^0_{\Higgs}(\oX_{\hkinf},-)$ by passing to a finite etale cover to make the monodromy of $\bL$ around $D$ unipotent. Denote the rank of $\bL$ by $d$ and let $\pi: Y\to X$ be the finite $G:=GL_d(\bZ/p^d)$-Galois cover of $X$ that trivializes $\bL/p^{d}$. 

 Let $\pi':Y'\to \oX$ be the normalization of $\oX$ in $Y$. The action of $G$ extends on $Y'$ and $\pi'$ is a $G$-equivariant finite morphism. Applying functorial resolution of singularities \cite[Theorem 1.1.13(iii)]{temkin} we can find a diagram 
\[
\begin{tikzcd}
Y\arrow[r]\arrow[d,"\pi"] &\oY\arrow[d,"\opi"] \\
X\arrow[r] & \oX
\end{tikzcd}
\] 

\noindent such that $\oY$ is a proper smooth rigid-analytic variety, the complement $D'=\oY\setminus Y$ is a normal crossings divisor, and $G$-action extends from $Y$ to $\oY$ and $\opi$ is $G$-equivariant. Since every quasi-unipotent element of $1+p^d\Mat_d(\bZ_p)$ is unipotent, $\pi^*\bL$ has unipotent monodromy around $D'$. The maps $\alpha_{\bL}$ and $\alpha_{\pi^*\bL}$ induce the diagram 

\begin{equation}\label{unipotent comparison formula}
\begin{tikzcd}
\End_{X_{\oK}}(\bL)\otimes_{\bZ_p}C\arrow[r,"\sim"]\arrow[d,"\sim"] & H^0_{\Higgs}(\oX_{\hkinf},\cH^{\log}(\bL\otimes_{\bZ_p}\bL^{\vee}))\otimes_{\hkinf}C\arrow[r, hookrightarrow, "\alpha_{\bL}"]\arrow[d] & \End_{\oX_{\hkinf},\Higgs}\cH^{\log}(\bL)\otimes_{\hkinf}C\arrow[d, hookrightarrow] \\
\End_{Y_{\oK}}(\pi^*\bL)^G\otimes_{\bZ_p}C\arrow[r,"\sim"] & H^0_{\Higgs}(\oY_{\hkinf},\cH^{\log}(\pi^*(\bL\otimes_{\bZ_p}\bL^{\vee})))^G\otimes_{K_{\hkinf}}C \arrow{r}{\alpha_{\bL}}[swap]{\sim} & \End_{\oY_{\hkinf},\Higgs}\cH^{\log}(\pi^*\bL)^G\otimes_{\hkinf}C
\end{tikzcd}
\end{equation}

The first arrows in both rows are isomorphisms by \cite[Theorem 3.2.4(2)]{dllz}, the second one in the second row is by Theorem 3.2.12(1) of loc. cit. and the leftmost vertical one is an isomorphism because $\pi:Y_{\oK}\to X_{\oK}$ is an etale $G$-Galois cover. The second arrow in the first row and the rightmost vertical one are both injections, so the composition of the first row is an isomorphism, as desired.
\end{proof}

\begin{rem}
In part (iii) the condition on the existence of a $\bZ_p$-lattice is necessary. Consider a Tate elliptic curve $X=\oX=\bG_m/\langle \gamma\rangle$ defined by quotienting $\bG_{m,K}=(\Spec K[x,x^{-1}])^{\an}$ by the automorphism $\gamma$ defined as $\gamma^*(x)=px$. Define a rank $1$ local system $\bM$ on $X$ by taking the constant local system $\ubQp$ on $\bG_m$ and equipping it with the descent datum $\gamma^*\ubQp\xrightarrow{p}\ubQp$. Then $H^0(X_{\oK},\bM)=0$ while $H^0(X_{\hkinf},\cH(\bM))^{\theta=0}=H^0_{\proet}(X_{\hkinf},\bM\otimes_{\bQ_p}\widehat{\cO})=H^0_{\proet}(\bG_{m,\hkinf},\widehat{\cO})^{\gamma=p^{-1}}=\hkinf\cdot x^{-1}$. In particular, the comparison between etale and Higgs cohomology fails already for cohomology in degree $0$ and (iii) is false e.g. for $\bL=\ubQp\oplus\bM$.
\end{rem}

This construction translates the issue of proving that a local system $\bL$ is de Rham into a question about the action of $\Gamma$ on $\cR\cH^{\log}(\bL)$. We will analyze this action using Fontaine-Sen's technique of decompletion. 

\section{Decompleted Simpson and Riemann-Hilbert correspondences}
\newcommand{\Y}{Y}
\newcommand{\CY}{\cY}

The decompleted versions of $\cH(\bL)$ and $\cR\cH(\bL)$ we need have essentially been constructed by Koji Shimizu. Sections 3.2 and 4.1 of \cite{kht} construct those over affinoid bases where the bundle $\cH(\bL)$ is trivial. We check in this section that these constructions can be globalized to work over arbitrary smooth quasi-compact rigid-analytic varieties, the main result of this section is Proposition \ref{main decompleted rh}. 

We will start with a general discussion of decompletions in relative setting that will later be specialized to the cases of $\cH^{\log}(\bL)$ and $\cR\cH^{\log}(\bL)$. For the duration of this section let $\Y$ be a smooth quasi-compact rigid-analytic variety over $K$ which we will view as an adic space over $\Spa(K,\cO_K)$. Consider the sheaves $\cO_{{\Y}_{\Kinf}}:= \cO_{\Y}\otimes_K \Kinf$, $\cO_{{\Y}_{\Kinf}}[[t]]:=\lim\limits_{\leftarrow}\cO_{\Y}\otimes_K \Kinf[t]/t^n$, and $\cO_{{\Y}_{\Kinf}}{((t))}:=\cO_{{\Y}_{\Kinf}}{[[t]]}\otimes_{\Kinf[[t]]}\Kinf((t))$. Denote by ${\Y}_{\Kinf}$, $\Y_{\Kinf[[t]]}$, and ${\Y}_{\Kinf((t))}$ the ringed spaces $({\Y}, \cO_{{\Y}_{\Kinf}})$, $(\Y,\cO_{\Y_{\Kinf}}{[[t]]})$, and $({\Y}, \cO_{{\Y}_{\Kinf}}{((t))})$ respectively. We will call locally free sheaves of finite rank on these ringed spaces simply 'vector bundles'. Recall that ${\Y}_{\hkinf}$ is the base change ${\Y}\times_{\Spa(K,\cO_K)}\Spa(\hkinf,\cO_{\hkinf})$ and $\CY^+$, $\CY$ are the ringed spaces $({\Y}_{\hkinf},\cO_{\Y}\widehat{\otimes}_K L^+_{\dR})$ and $({\Y}_{\hkinf},\cO_{\Y}\widehat{\otimes}_K L_{\dR})$ respectively.

Note that it is not immediate that a vector bundle on some $\Y_{K_n}$ gives rise to a vector bundle on $\Y_{\Kinf}$ because the underlying topological spaces $|\Y_{K_n}|$ and $|\Y|$ are different. The next lemma, however, shows that one can always find a cover of $\Y$ that trivializes a vector bundle on $\Y_{K_n}$.

\begin{lm}\label{opens descent}
(i) For any quasi-compact open $U\subset \Y_{\hkinf}$ there exists $n\in\bN$ and an open $U_n\subset \Y_{K_n}$ such that $U=U_n\times_{K_n}\hkinf$.

(ii) If $E$ is a vector bundle on $\Y_{K_n}$ for some $n\in\bN$ then there exists a finite cover of $\Y=\bigcup\limits_i U_i$ such that the cover $\{U_{i,K_n}\}_{i\in I}$ of $\Y_{K_n}$ trivializes $E$.
\end{lm}

\begin{proof}
(i) Since it is enough to prove the statement of each member of a finite cover $U=\bigcup\limits_{i\in I}V_i$ we may assume that $\Y$ is affinoid and $U$ is a rational subset. The claim is then given by \cite[Lemma 2.1.3(ii)]{berkovich}.

(ii) As $\Y$ is assumed to be quasi-compact, we may assume that it is moreover affinoid $\Y=\Spa(A,A^+)$ so $E$ is associated to a finite projective module $M$ over $A_{K_n}$. We will prove the statement by induction on the maximum of ranks of restrictions of $E$ to connected components of $\Y_{K_n}$ (we cannot assume that $\Y_{K_n}$ is connected even if it is given that $\Y$ is connected). It is enough to find a cover $\Y=\bigcup\limits_i U_i$ by affinoids $U_i=\Spa(A_i,A_i^+)$ such that on each $U_{i,K_n}$ there is a section $s:\cO_{U_{i,K_n}}\to E|_{U_{i,K_n}}$ that does not vanish at the points of connected components of $U_{i,K_n}$ where $E$ is not identically zero. The exact sequences $0\to\cO_{U_{i,K_n}}\xrightarrow{s} E|_{U_{i,K_n}}\to \coker(s)\to 0$ split automatically by Tate's acyclicity and we can find a trivializing cover for $\coker(s)$ by induction assumption. For a point $x\in |\Y|$ consider the finite set $\{x_1,\dots, x_r\}=\pi_n^{-1}(x)\subset |\Y_{K_n}|$ where $\pi_n:\Y_{K_n}\to \Y$ is the canonical projection. For each $x_i$ such that $E$ is not zero on the connected component containing $x_i$ we can find a section $s_i:A_{K_n}\to M$ that does not vanish on some $x_i\in V_i\subset \Y_{K_n}$ and a function $f_i\in A$ that does not vanish in the residue field $k(x_i)$ but vanishes in each $k(x_j)$ for $j\neq i$. If $E$ vanishes on the connected component containing $x_i$ we declare $s_i=0$. Let $\tV\subset \Y_{K_n}$ be union of the open subset where the section $\sum\limits_{i=1}^r f_is_i:A_{K_n}\to M$ does not vanish and all the connected components of $\Y_{K_n}$ where $E$ is zero. Take $V\subset\Y_{K_n}$ to be the intersection of its $\Gal(K_n/K)$-orbit. Then $V=\pi_n^{-1}(U)$ for some open $U\subset\Y$ and $E|_{U_{i,K_n}}$ admits a section with desired property.
\end{proof}

\begin{pr}\label{vb decompletion}
Given a $\Gamma$-equivariant vector bundle $E$ on $\Y_{\hkinf}$, there is a $\Gamma$-equivariant vector bundle $E^{\fin}$ on $\Y_{\Kinf}$ with an isomorphism $E^{\fin}_{\hkinf}\simeq E$. For an affinoid open $\Spa(A,A^+)=U\subset \Y$ the sections of $E^{\fin}$ are given by:
\begin{equation}\label{plain decompletion formula}
E^{\fin}(U)=\{s\in E(U_{\hkinf})\text{ such that the }A\text{-span of the }\Gamma\text{-orbit of }s\text{ is finitely generated over }A\}
\end{equation}

Moreover, there is an endomorphism $\phi:E^{\fin}\to E^{\fin}$ such that for any section $s\in E^{\fin}(U_{\Kinf})$ we have \begin{equation}\label{sen formula}\gamma(s)=\exp(\phi\cdot \log\chi_{\cycl}(\gamma))s\end{equation} for all $\gamma$ in some open subgroup $\Gamma_s$ of $\Gamma$. The association $E\mapsto (E^{\fin},\phi)$ is functorial with respect to morphisms of equivariant vector bundles on $\Y_{\hkinf}$ and is compatible with tensor products.
\end{pr}

\begin{proof}
Denote by $E^{\fin}_0$ the presheaf of $\cO_{\Y}\otimes_K \Kinf$-modules defined on affinoid opens in $\Y$ by (\ref{plain decompletion formula}).

First, let us prove that $E^{\fin}_0$ satisfies the sheaf condition: if $U=\bigcup\limits_{i\in I}U_i$ is an open covering of an affinoid by affinoids, we need to check that a section $s\in E(U_{\hkinf})$ lies in $E^{\fin}_0(U)$ if all its restrictions to $U_{i,{\Kinf}}$ lie in the corresponding submodules. We may assume that the covering has finitely many opens as $U$ is quasi-compact. For each open $U_i=\Spa(A_i,A_i^+)$ there is a finite subset $\Gamma_i\subset \Gamma$ such that the finite set $\Gamma_i\cdot s|_{U_i}$ generates the module $\langle\Gamma\cdot s|_{U_i}\rangle$ over $A_i$. Taking the union of these $\Gamma_i$ for all $i$ we obtain a finite set of elements $\gamma_1,\dots,\gamma_N\in \Gamma$ such that for all $i$ the inclusion $\langle \gamma_1(s)|_{U_i},\dots,\gamma_N(s)|_{U_i}\rangle_{A_i}\subset\langle \Gamma\cdot s|_{U_i}\rangle_{A_i}$ is equality. Since $E$ is, in particular, a coherent sheaf on $\Y_{\hkinf}$, we have $E(U_i)=E(U)\otimes_{A_{\hkinf}}A_{i,\hkinf}=E(U)\otimes_{A}A_i$. The map $A\to A_i$ is flat so for any set $I\subset E(U)$ we have $\langle x|x\in I\rangle_A\otimes_A A_i\simeq \langle x|_{U_i}|x\in I\rangle_{A_i}$ under the canonical map. Therefore $\langle \gamma_1(s),\dots,\gamma_N(s)\rangle_{A}\subset \langle\Gamma\cdot s\rangle_A$ becomes an equality after taking the tensor product with any $A_i$ over $A$. Since the map $A\to\prod\limits_{i\in I}A_i$ is faithfully flat, this inclusion is in fact equality, as desired.

By \cite[Theorem 2.9]{kht} for any affinoid open $U\subset \Y$ such that $E|_{U_{\hkinf}}$ is trivial as a vector bundle we have $E^{\fin}_0(U)\otimes_{\cO(U_{\Kinf})}{\cO(U_{\hkinf})}\simeq E(U_{\hkinf})$. By Lemma \ref{opens descent} (i) there exists $n$ and a finite cover of $\Y_{K_n}$ such that its base change to $\Y_{\hkinf}$ trivializes $E$ so $E^{\fin}_0$ defines a locally free sheaf of $\cO_{\Y_{\Kinf}}$-modules on ${\Y}_{K_n}$.  Possibly after increasing $n$, there is a vector bundle $E_n$ on ${\Y}_{K_n}$ such that $E_n\otimes_{K_n}\Kinf\simeq E_0^{\fin}$. By Lemma \ref{opens descent} (ii), the sheaf $E_n$ and therefore $E^{\fin}_0$ is locally free on $\Y$, as desired.

The local construction of the operator $\phi$ from  \cite[Proposition 2.20]{kht} then immediately globalizes because $\phi$ is uniquely characterized by (\ref{sen formula}). The last assertion about functoriality and tensor compatibility follows from the base change functor from vector bundles on $\Y_{\Kinf}$ to $\Y_{\hkinf}$ being faithful and conservative.
\end{proof}

\begin{definition}
A filtration on a vector bundle $E$ on $\Y_{\Kinf((t))}$ is a collection $\{F^iE\}_{i\in\bZ}$ of locally free $\cO_{\Y_{\Kinf}}[[t]]$-subsheaves $F^iE\subset E$ such that $F^{i+1}E\subset F^iE$ and $t\cdot F^iE=F^{i+1}E$ for all $i$. A filtered vector bundle with integrable connection on $\Y_{\Kinf((t))}$ is a filtered vector bundles $(E, \{F^iE\}_{i\in\bZ})$ as above together with an integrable $\Kinf((t))$-linear connection $\nabla:E\to E\otimes_{\cO_{\Y_{\Kinf}}((t))}\Omega^1_{\Y_{\Kinf}/\Kinf}((t))$ such that $\nabla(F^iE)\subset F^{i-1}E\otimes_{\cO_{\Y_{\Kinf}}[[t]]}\Omega^1_{\Y_{\Kinf}/\Kinf}[[t]]$. 
\end{definition}

Note that given a vector bundle $E$ on $\Y_{\Kinf[[t]]}$ we can attach to it a vector bundle on $\CY^+$ as follows. Suppose that $\Y=\bigcup\limits_{i\in I} U_i$ is a trivializing cover for $E$ and define a vector bundle $E_{L_{\dR}^+}$ on $\CY^+$ trivialized by the cover $\bigcup\limits_{i\in I}|U_{i,\hkinf}|=|\CY^+|$ using the gluing data obtained from that for $E$ applying the maps $\cO_{U_{ij}}[[t]]\to \cO_{U_{ij}}\widehat{\otimes}_K L_{\dR}^+$.

\begin{pr}\label{de rham decompletion}
Suppose that $E$ is a $\Gamma$-equivariant vector bundle on $\CY^{+}$. Then there is a $\Gamma$-equivariant vector bundle $E^{\fin}$ on $\Y_{\Kinf[[t]]}$ together with an isomorphism $E^{\fin}_{L_{\dR}^+}\simeq E$. For every $n$ the sections of $E^{\fin}/t^n$ on an affinoid open $\Spa(A,A^+)=U\subset \Y$ are given by:
\begin{equation}\label{de rham decompletion formula}
(E^{\fin}/t^n)(U)=\{s\in E(\cU^+_{L_{\dR}^+/t^n})\text{ such that the }A\text{-span of the }\Gamma\text{-orbit of }s\text{ is finitely generated over }A\}
\end{equation}

Moreover, there is an endomorphism $\phi:E^{\fin}\to E^{\fin}$ that is $\cO_{\Y_{\Kinf}}$-linear and satisfies $\phi(ts)=t\phi(s)+ts$ for any local section $s\in E^{\fin}(U)$. For any $n$ and $s\in E^{\fin}(U)$ \begin{equation}\label{de rham sen formula}\gamma(s)=\exp(\phi\cdot \log\chi_{\cycl}(\gamma))s\end{equation} for all $\gamma$ in some open subgroup $\Gamma_{n,s}$ of $\Gamma$. The association $E\mapsto (E^{\fin},\phi)$ is functorial with respect to morphisms of equivariant vector bundles on $\CY^+$ and is compatible with tensor products.
\end{pr}

\begin{proof}
Define the presheaf of $\Kinf[t]/t^n$-module on affinoid open subspaces of $\Y$ by $$E^{\fin}_n(U)=\{s\in E(\cU^+_{L_{\dR}^+/t^n})\text{ such that the }A\text{-span of }\Gamma\text{-orbit of }s\text{ is finitely generated over }A\}.$$ There is a left exact sequence of presheaves $0\to E_{n-1}^{\fin}\xrightarrow{t}E_{n}^{\fin}\to (E/t)^{\fin}$. We will prove by induction on $n$ that $E_n^{\fin}$ is a sheaf, the base case $n=1$ being covered by Proposition \ref{vb decompletion}. Suppose that $U=\bigcup\limits_{i\in I}U_i$ is a finite cover of an affinoid by affinoids and a section $s\in E(\cU^+_{L_{\dR}^+/t^n})$ is such that every restriction $s|_{U_i}$ lies in $E_n^{\fin}(U_i)$. We want to conclude that $s$ lies in $E^{\fin}_n(U)$. Denote by $\os$ the image of $s$ in $(E/t^{n-1})(U_{\hkinf})$. By the inductive assumption, $E_{n-1}^{\fin}$ is a sheaf so there is a finite set of elements $\gamma_1,\dots, \gamma_N\in\Gamma$ such that $\gamma_1\cdot \os,\dots\gamma_N\cdot \os$ span $\langle\Gamma\cdot\os\rangle_{A}$ inside $E(\cU_{L_{\dR}^+/t^{n-1}})$. This automatically implies that $\gamma_1\cdot s,\dots\gamma_N\cdot s,t\gamma_1(s),\dots t\gamma_N(s)$ span $\langle \Gamma\cdot s\rangle_A$ inside $E(\cU^+_{L_{\dR}^+/t^n})$. Therefore, $E^{\fin}_n$ is a sheaf.

Next, by \cite[Proposition 2.11]{kht} for any affinoid $U\subset X$ such that $E(\cU^+_{L_{\dR}^+/t^n})$ is a free $A\widehat{\otimes}_{K}L_{\dR}^+/t^n$-module the subspace $E_{n}^{\fin}(U)$ is a free $\Kinf[t]/t^n$-module of the same rank and the canonical map $E_{n}^{\fin}(U)\widehat{\otimes}_{\Kinf[t]/t^n}L_{\dR}^+/t^n\simeq E(\cU^+_{L_{\dR}^+/t^n})$ is an isomorphism.

Finally, there is a cover of $\Y_{\hkinf}$ by affinoids that trivializes $E$, it descends to some $K_m$ by Lemma \ref{opens descent}(i) so $E_n^{\fin}$ defines a locally free $\cO_{\Y_{\Kinf}}{[t]/t^n}$-module on $\Y_{K_m}$. After increasing $m$ this module may be non-canonically descended to a locally free sheaf of $\cO_{\Y_{K_m}}[t]/t^n$-modules on $\Y_{K_m}$ and Lemma \ref{opens descent}(ii) provides a cover of $\Y$ whose base change to $K_m$ trivializes this locally free sheaf. This cover therefore also trivializes the $\cO_{\Y_{\Kinf}}[t]/t^n$-module $E_n^{\fin}$ and we define the desired $E^{\fin}$ as $\lim\limits_n E_n^{\fin}$. \cite[Proposition 2.24]{kht} gives us the desired Sen endomorphism $\phi$ over the trivializing cover and it glues to a global endomorphism because of the characterizing property (\ref{de rham sen formula}).
\end{proof}

Now let us put ourselves in the setting of Section \ref{section2}: $\oX$ is a geometrically connected smooth proper rigid-analytic variety over $K$ and $X\subset \oX$ is an open subvariety such that $D=\oX\setminus X$ is a normal crossings divisor. Applying these decompletion results to $\cR\cH^{\log}(\bL)$ for a local system $\bL$ on $X$ gives

\begin{pr}\label{main decompleted rh}
To every $\bQ_p$-local system $\bL$ on $X$ we can assign a filtered vector bundles with integrable connection $\RH^{\log}(\bL)$ on $\oX_{\Kinf((t))}$ that comes equipped with a $\Kinf$-linear endomorphism $\phi$ preserving the filtration and commuting with the connection.

(i) Denote the graded piece of the filtration $\gr^0\RH^{\log}(\bL)$ by $\Hig^{\log}(\bL)$. All the other graded pieces $\gr^i\RH^{\log}(\bL)$ are then canonically identified with $\Hig^{\log}(\bL)$ as well. $\Hig^{\log}(\bL)$ is a vector bundle of rank $\rk_{\bQ_p}\bL$ on $\oX_{\Kinf}$ and connection induces a Higgs field $\theta:\Hig^{\log}(\bL)\to\Hig^{\log}(\bL)\otimes\Omega^1_{\oX_{\Kinf}/\Kinf}$. The endomorphism $\phi$ defines an endomorphism $\phi$ of the vector bundle $\Hig^{\log}(\bL)$ such that $\theta\circ \phi=(\phi-1)\circ \theta$.

(ii) The characteristic polynomial $\det(T\cdot \Id_{\Hig^{\log}(\bL)}-\phi)$ has coefficients in $K$. 

(iii) If $\bL$ admits a $\bZ_p$-lattice then there is an $H_K$-equivariant isomorphism of $C$-algebras $$\End_{X_{\oK}}({\bL})\otimes_{\bQ_p}C\simeq \End_{\oX_{\Kinf},\Higgs}(\Hig^{\log}(\bL))\otimes_{\Kinf}C.$$
\end{pr}

\begin{proof}
Start with $\cR\cH^{\log}(\bL)$ provided by Theorem \ref{liu zhu rh} and apply Proposition \ref{de rham decompletion} to $E=\Fil^0\cR\cH^{\log}(\bL)$ with $Y=\oX$. Define $\RH^{\log}(\bL)$ as $E^{\fin}[1/t]$. By construction, its $\cO_{X}\otimes_K L_{\dR}$-rank is equal to $\rk_{\bQ_p}\bL$ and is equipped with an endomorphism $\phi$ related to the action of $\Gamma$ on $\cR\cH^{\log}(\bL)$ via (\ref{de rham decompletion formula}).

For part (i), the identification $\gr^i\cR\cH^{\log}(\bL)\simeq \cH^{\log}(\bL)(i)$ induces $\gr^i\RH^{\log}(\bL)\simeq \Hig^{\log}(\bL)$ where we use our choice of $t$ to identify $\cH^{\log}(\bL)(i)$ with $\cH^{\log}(\bL)$. The Sen endomorphism  on $\gr^{-1}\RH^{\log}(\bL)$ is given by $\phi-1$ in terms of this identification where $\phi$ is the Sen endomophism on $\Hig^{\log}(\bL)$. This gives the equality $\theta\circ \phi=(\phi-1)\circ \theta$ of maps from $\Hig^{\log}(\bL)$ to $\Hig^{\log}(\bL)\otimes_{\cO_{\oX_{\Kinf}}}\Omega^1_{\oX_{\Kinf}/\Kinf}$. 

In part (ii), the fact the coefficients of the characteristic polynomial of $\phi$ are constant (that is, lie in $\oK$) across $\oX$ is \cite[Theorem 4.8]{kht}. It is, in fact, a purely local statement and is a consequence of the fact that the Higgs bundle $\Hig^{\log}(\bL)$ together with the endomorphism $\phi$ appears as a graded piece of a vector bundle with an integrable connection, but in our setting of a proper $\oX$ it also follows immediately from all global functions being constant $H^0(\oX_{\Kinf},\cO_{X_{\Kinf}})=\Kinf$ (recall that $X$ is assumed to be geometrically connected, otherwise the statement (ii) is false as the example of $X=\Spa(K',\cO_{K'})$ for a finite extension $K'\supset K$ shows).

It remains to prove that the coefficients of $\det(T\cdot \Id_{\Hig^{\log}(\bL)}-\phi)$ in fact lie in $K$. If $X$ has a $K$-point, this follows from \cite[Theorem 5]{sen} applied to the stalk of $\bL$ at that point. To cover the case of $X$ not admitting a rational point, we can repeat the proof of Sen's result in our relative setting: the Sen endomorphism $\phi$ induces an $\hkinf$-linear endomorphism of $\cH^{\log}(\bL)$ that commutes with the semi-linear $\Gamma$-action on that vector bundle. Therefore the coefficients of the characteristic polynomial of $\phi$ lie in $H^0(\oX_{\hkinf},\cO_{\oX_{\hkinf}})^{\Gamma}=K$.

Part (iii) follows from Theorem \ref{liu zhu rh} (iii) together with the observation that $H^0(\oX_{\Kinf}, E^{\fin})=H^0(\oX_{\hkinf},E)^{\fin}$ for any $\Gamma$-equivariant vector bundle $E$ on $\oX$. 
\end{proof}
We will call the roots of the characteristic polynomial $\det(T\cdot \Id_{\Hig^{\log}(\bL)}-\phi)$ viewed as elemens of $\oK$ the {\it Hodge-Tate weights} of $\bL$. 

\begin{lm}\label{senoperator criteria}
(i) A local system $\bL$ is de Rham if and only if $\RH^{\log}(\bL)^{\phi=0}$ is a vector bundle of rank $\rk_{\bQ_p}\bL$ on $\oX_{\Kinf}$.

(ii) A local system $\bL$ is Hodge-Tate if and only if $\phi$ is a semi-simple endomorphism of $\Hig^{\log}(\bL)$ and Hodge-Tate weights of $\bL$ lie in $\bZ$.
\end{lm}

\begin{proof}
(i) By definition, a local system is de Rham if and only if the invariants $D_{\dR,\log}(\bL):=\cR\cH^{\log}(\bL)^{\Gamma}$ form a vector bundle on $X$ of rank $\rk_{\bQ_p}\bL$. Suppose that $\RH^{\log}(\bL)^{\phi=0}$ is a vector bundle of rank $\rk_{\bQ_p}\bL$ on $X_{\Kinf}$. Therefore there is a finite cover $\oX=\bigcup\limits_i U_i$ such that for each $U_i$ the kernel $\RH^{\log}(\bL)(U_i)^{\phi=0}$ is spanned by $\rk_{\bQ_p}\bL$ elements $s_{i1},\dots, s_{id}$ with $\phi(s_{ij})=0$ for all $j$. By (\ref{de rham decompletion formula}) each $s_{ij}$ is stabilized by an open subgroup of $\Gamma$. Therefore for large enough $n$ they all are stabilized by $\Gamma_n:=\Gal(\Kinf/K_n)$ so $\rk_{\cO_{X_{K_n}}}D_{\dR}(\bL|_{X_{K_n}})\geq \rk_{\bQ_p}\bL$. These ranks are hence equal and $\bL|_{X_{K_n}}$ is de Rham which implies that $\bL$ itself is de Rham.

Conversely, suppose that $\bL$ is de Rham. By \cite[Theorem 3.2.7(2)]{dllz}  the sheaf $D_{\dR,\log}(\bL)=\cR\cH^{\log}(\bL)^{\Gamma}$ is a locally free sheaf of rank $\rk_{\bZ_p}\bL$. The canonical injective map $D_{\dR,\log}(\bL)\otimes_K \Kinf\to \RH^{\log}(\bL)^{\phi=0}$ is an isomorphism because the map $D_{\dR,\log}(\bL)\widehat{\otimes}_K{L_{\dR}}\to\cR\cH^{\log}(\bL)$ is.  In particular, $\RH^{\log}(\bL)^{\phi=0}$ is a vector bundle of rank $\rk_{\bQ_p}\bL$ on $\oX_{\Kinf}$.

(ii) The operator $\phi$ on $\Hig^{\log}(\bL)$ is semi-simple with integer eigenvalues if and only if its restriction $\Hig^{\log}(\bL)|_X$ is. The latter condition is a restatement of what it means for $\bL$ to be Hodge-Tate by \cite[Theorem 5.5]{kht}.
\end{proof}

\section{Geometrically irreducible local systems are Hodge-Tate up to a twist}

As before, $X$ is a geometrically connected smooth rigid-analytic variety over $K$ embedded as a Zariski open subvariety into a smooth proper rigid-analytic variety $\oX$ such that $\oX\setminus X$ is a normal crossings divisor. 

\begin{thm}\label{ht main}
Suppose that $\bL$ is a $\obQ_p$-local system on $X$ that admits a $\obZ_p$-lattice. If $\bL|_{X_{\oK}}$ is irreducible then there exists a character $\chi:G_{K}\to\obQ_p^{\times}$ such that $\bL\otimes\chi$ is Hodge-Tate.
\end{thm}

The next lemma says that proving the theorem amounts to proving that the projective representation of the fundamental group corresponding to $\bL$ is Hodge-Tate. In the case where $K$ is a finite extension of $\bQ_p$ this statement follows from \cite[Corollary 3.2.12]{patrikis}. 

Let us introduce some notation for working with the Sen endomorphism associated to a local system $\bL$ of $E$-vector spaces, for a finite extension $E\supset \bQ_p$. Viewing $\bL$ as a $\bQ_p$-local system (of rank $\rk_{E}\bL\cdot[E:\bQ_p]$) we may consider the associated Higgs bundle $\Hig^{\log}(\bL)$ on $\oX_{\Kinf}$ that has a structure of an $E$-module, by functoriality. In particular, $\Hig^{\log}(\bL)$ carries an action of the algebra $E\otimes_{\bQ_p}K$. This algebra can be decomposed as $E\otimes_{\bQ_p}K=\bigoplus\limits_{\tau:E\to \oK}K_{\tau}$ where $\tau$ runs over all $\bQ_p$-linear embeddings of $E$ into $\oK$ and $K_{\tau}$ is the subfield of $\oK$ generated by $K$ and $\tau(E)$. This decomposition induces a decomposition of every $E\otimes_{\bQ_p}K$-module. In particular, we get a decomposition $\Hig^{\log}(\bL)=\bigoplus\limits_{\tau}\Hig^{\log}(\bL)_{\tau}$ where $\Hig^{\log}(\bL)_{\tau}$ is defined as $\Hig^{\log}(\bL)\otimes_{E\otimes_{\bQ_p}K}K_{\tau}$. This decomposition is respected by the Sen endomorphism $\phi$ and we denote by $\phi_{\tau}$ its component on the summand $\Hig^{\log}(\bL)_{\tau}$. For the duration of the proofs of Lemmas \ref{projective lift} and \ref{hodge-tate: character lemma} we will enhance the notation $\phi_{\tau}$ to $\phi_{\bL,\tau}$ as the discussion will involve Sen endomorphisms of various local systems. We will refer to eigenvalues of $\phi_{\bL,\tau}$ as {\it $\tau$-Hodge-Tate weights} of $\bL$.

\begin{lm}\label{projective lift}
Let $\ox$ be arbitrary geometric point of $X$ and $E\supset\bQ_p$ be a finite extension. If $\rho:\pi_1^{\et}(X,\ox)\to GL_d(\cO_E)$ is a continuous representation such that the corresponding projective representation $\rho_{\proj}:\pi_1^{\et}(X,\ox)\to PGL_d(\cO_E)$ is Hodge-Tate then there exists a character $\chi:G_K\to\obQ_p^{\times}$ such that $\rho\otimes\chi$ is Hodge-Tate.
\end{lm}

\begin{proof}
Denote by $\bL$ the $E$-local system corresponding to $\rho$. By definition, $\rho_{\proj}$ being Hodge-Tate means that the adjoint local system $\bL\otimes_E\bL^{\vee}$ is Hodge-Tate, cf. Example \ref{projective rep example}. 

The assumption that $\bL\otimes_E\bL^{\vee}$ is Hodge-Tate translates into the fact that the endomorphism $\phi_{\bL}\otimes 1-1\otimes \phi_{\bL}^{\vee}$ of $\Hig^{\log}(\bL)\otimes_E\Hig^{\log}(\bL)^{\vee}$ is semi-simple with integer eigenvalues. Here we have used that the functor $\Hig^{\log}|_{X_{\Kinf}}$ is monoidal together with the observation that semi-simplicity and integrality of eigenvalues can be checked after restricting to a dense open in $\oX_{\Kinf}$. In the notation of the paragraph preceeding the statement of the lemma, this is equivalent to saying that for every $\tau$ the Sen endomorphism $\phi_{\bL,\tau}\otimes 1-1\otimes\phi_{\bL,\tau}^{\vee}$ on $\Hig^{\log}(\bL)_{\tau}\otimes_{\cO_{\oX_{\Kinf^{\tau}}}}\Hig^{\log}(\bL)^{\vee}_{\tau}$ is semi-simple with eigenvalues in $\bZ$. This means that for every $\tau$ the operator $\phi_{\bL,\tau}$ is semi-simple and there exists an element $a_{\tau}\in\oK$ such that the eigenvalues of $\phi_{\bL,\tau}$ lie in $a_{\tau}+\bZ$.

Our goal now is to find a character $\chi:G_K\to E'^{\times}$ for a finite extension $E'\supset E$ such that the Sen endomorphism corresponding to the $E'$-local system $\bL\otimes_E \chi$ has integral eigenvalues. This would be immediate if $K$ was contained in $\obQ_p$ as then \cite[Lemma 2.1.3]{matteo} can be used to provide a character $\chi$ with the eigenvalue of $\phi_{\chi,\tau}$ on $\Hig(\chi)_{\tau}$ equal to $-a_{\tau}$ for every $\tau$. We don't know if the analog of \cite[Lemma 2.1.3]{matteo} is true for an arbitrary $K$ but we can nevertheless construct $\chi$ as follows.

Suppose that eigenvalues of $\phi_{\bL,{\tau}}$ on $\Hig^{\log}(\bL)_{\tau}$ are $a_{\tau}+n_{\tau,1},\dots, a_{\tau}+n_{\tau, d}$ for some $n_{\tau,j}\in\bZ$. Then the eigenvalue of $\phi_{\det\bL,\tau}$ on $\Hig^{\log}(\det\bL)_{\tau}$ is equal to $d\cdot a_{\tau}+\sum\limits_{j=1}^d n_{\tau,j}$. To get started on the construction of $\chi$, we would like to have a character $\chi_0:G_K\to (E')^{\times}$ with values in a finite extension $E'\supset E$, such that the eigenvalue of $\phi_{\chi_0,\tau}$ on $\Hig(\chi_0)_{\tau}$ is equal to that of $\phi_{\det\bL,\tau}$ on $\Hig^{\log}(\det\bL)_{\tau}$. If $X$ has a $K$-point $x$ then such $\chi_0$ is provided simply by $(\det\bL)_x$. In general, if $X$ acquires a rational point $x\in X(K')$ after a finite extension $K'\supset K$ we define $\chi_0$ as the output of Lemma \ref{hodge-tate: character lemma} (ii) applied to $\eta'=(\det \bL)_x$, the assumption of this lemma being satisfied by Proposition \ref{main decompleted rh} (ii).

Next, we will construct a character $\chi_1:G_K\to E^{\times}$ such that the $\tau$-Hodge-Tate weight of $\chi_{1,\tau}$ is the integer $\sum\limits_{j=1}^d n_{\tau,j}$, for every $\tau$. Let $L$ be the field $\obQ_p\cap K\subset \oK$. The inclusion $\oL=\obQ_p\subset \oK$ defines a map of Galois groups $\alpha:G_K\to G_L$ and, if $\eta:G_L\to E^{\times}$ is a character, the $\tau$-weight of $\eta$ coincides with the $\tau$ weight of $\eta\circ \alpha$, for every embedding $\tau:E\to\oK$ (all embeddings necessarily factor through $\oL$). The character $\chi_1$ can be now obtained by precomposing with $\alpha$ the character provided by \cite[Lemma 2.1.3]{matteo} applied to the field $L$ and the set of weights $(\sum\limits_{j=1}^d n_{\tau,j})_{\tau}$.

The desired character $\chi$ can now be defined as the "$d$-th root" of $\chi_1\otimes\chi_0^{-1}$, provided by Lemma \ref{hodge-tate: character lemma}(i).
\end{proof}

\begin{lm}
\label{hodge-tate: character lemma}
\begin{enumerate}[(i)]
\item For any character $\eta:G_K\to (E')^{\times}$ and an integer $d$ there exists a character $\eta':G_K\to E'^{\times}$ with values in a finite extension $E'\supset E$ such that $(\eta')^d$ is equal to $\eta$ when restricted to an open subgroup of $G_K$.

\item Suppose that $K'\supset K$ is a finite extension and $\eta':G_{K'}\to E^{\times}$ is a character such that for every $\tau:E\to \oK=\oK'$ the $\tau$-Hodge-Tate weight of $\eta'$ lies in $K_{\tau}$ (rather than in a larger field $K'_{\tau}$). Then there exists a character $\eta:G_K\to (E')^{\times}$ with values in a finite extension $E'\supset E$ such that the $\tau$-Hodge-Tate weight of $\eta$ is equal to the $\tau|_E$-Hodge-Tate weight of $\eta'$, for every $\tau:E'\to\oK$.
\end{enumerate}
\end{lm}

\begin{proof}
(i)  A very similar argument to the one we are going to give appears in the proof of \cite[Lemma 2.1.3]{matteo}. We will prove that there exists a continuous homomorphism $r_d:\cO_{E}^{\times}\to\obQ_p^{\times}$ that is given by the formula $r_d(x)=(1+x)^{1/d}:=\sum\limits_{i=0}^{\infty}\binom{1/d}{i}x^i$ for $x$ from an open subgroup of $\cO_{E}^{\times}$. Indeed, the topological group $\cO_E^{\times}$ can be decomposed as a product $\cO_E^{\times}=(\cO_E^{\times})_{\mathrm{tors}}\times U_E$ where $U_E$ is an open subgroup isomorphic to $\bZ_p^{[E:\bQ_p]}$. We can define a homomorphism $r'_d:U_E\to\obQ_p^{\times}$ by the formula $r'_d(x)=(1+x)^{1/d}$ on an open subgroup of $U_E$ where the series $(1+x)^{1/d}$ converges and then extend it arbitrarily to all of $U_E$ using that $\obQ_p^{\times}$ is a divisible group. The desired $r_d$ is then obtained by precomposing $r'_d$ with the projection $\cO_{E}^{\times}\to U_E$.

The "$d$-th root" of any $\eta'$ can now be defined as $r_d\circ \eta'$ using that the image of $\eta'$ is contained in $\cO_{E}^{\times}\subset E^{\times}$.

(ii) Extending $K'$, we may assume that it is Galois over $K$. The finite Galois group $\Gal(K'/K)$ then acts on $G_{K'}$ by conjugation and, for an element $\sigma\in \Gal(K'/K)$ we denote by $\eta'^{\sigma}$ the precomposition of the character $\eta'$ with the automorphism induced by $\sigma$. By our assumption, the Hodge-Tate weights of $\eta'$ and $\eta'^{\sigma}$ coincide for every $\sigma$. Hence the $\tau$-Hodge-Tate weight of the character $\varepsilon':=\prod\limits_{\sigma\in \Gal(K'/K)}\eta'^{\sigma}$ is equal to $[K':K]$ times the $\tau$-weight of $\eta'$, for every $\tau$. The character $\varepsilon'$ visibly extends to the character $\varepsilon:=\det\Ind^{G_{K'}}_{G_K}\eta'$ of $G_K$ so the desired $\eta$ can be obtained as the "$[K':K]$th root" of $\varepsilon$.
\end{proof}

Let us recall the generalized eigenspace decomposition of vector bundles with respect to an endomorphism. Suppose that we are given a locally free finite rank sheaf $M$ on a ringed space $(Z,\cO_Z)$ equipped with an endomorphism $f:M\to M$ such that the characteristic polynomial $\chi_f(T):=\det(T\cdot \Id_M-f)\in H^0(Z,\cO_Z)[T]$ has coefficients and all of its roots in a field $K\subset H^0(Z,\cO_Z)$. If $\chi_f(T)$ decomposes as $\prod (T-\lambda_i)^{m_i}$ with $\lambda_1,\dots\lambda_r$ being pairwise distinct elements of $K$ then the algebra $K[T]/\chi_f(T)\simeq\bigoplus\limits_i K[T]/(T-\lambda_i)^{m_i}$ acts on $M$ and induces a decomposition $M=\bigoplus\limits_i M_{\lambda_i}$ where each $M_{\lambda_i}:=\ker (f-\lambda_i\cdot \Id_{M})^{m_i}$ is a locally free sheaf on $(Z,\cO_Z)$.
\begin{proof}[Proof of Theorem \ref{ht main}]
Let $E$ be a finite extension of $\bQ_p$ such that $\bL$ is a scalar extension of a local system with coefficients in $E$, we will denote this local system by the same symbol $\bL$. By Lemma \ref{projective lift} it is enough to prove that the projective representation of $\pi_1^{\et}(X,\ox)$ corresponding to $\bL$ is Hodge-Tate. That can be checked after the base change of $X$ and $\bL$ to a finite extension of $K$, we may therefore assume that any embedding $E\to\oK$ factors through $K$, and all the eigenvalues of the Sen endomorphism belong to $K$.

We have $\End_{E}\bL|_{X_{\oK}}=E\cdot\Id$. Viewing $\bL$ as a $\bQ_p$-local system, we can form a Higgs bundle $\Hig^{\log}(\bL)$ on $\oX_{\Kinf}$ of rank $[E:\bQ_p]\cdot\rk_{E}\bL$ equipped with an action of $E$. This action commutes with the Sen endomorphism $\phi$ of the bundle $\Hig^{\log}(\bL)$. By Proposition \ref{main decompleted rh} (iii) we have $(\End_{\oX_{\Kinf},\Higgs}\Hig^{\log}(\bL))\otimes_{\Kinf}\hkinf=(\End_{X_{\oK}}\bL\otimes_{\bQ_p}C)^{H_K}=E\otimes_{\bQ_p}C^{H_K}=E\otimes_{\bQ_p}\hkinf$ where the second equality follows from the fact that all endomorphisms of $\bL|_{X_{\oK}}$ extend to endomorphisms of $\bL$ over $K$, so the Galois action on the endomorphisms of $\bL|_{X_{\oK}}$ is trivial. Therefore, the map $E\to \End_{\oX_{\Kinf},\Higgs}\Hig^{\log}(\bL)$ induces an isomorphism $E\otimes_{\bQ_p}\Kinf\simeq\End_{\oX_{\Kinf},\Higgs}\Hig^{\log}(\bL)$.

As we have discussed at the beginning of this section, we get a decomposition $\Hig^{\log}(\bL)=\bigoplus\limits_{\tau}\Hig^{\log}(\bL)_{\tau}$. By the description of endomorphisms, for each $\tau$ we have that all endomorphisms of the Higgs bundle $\Hig^{\log}(\bL)_{\tau}$ are scalars from $\Kinf$. For a given $\tau$, suppose that $\prod\limits_i (t-\lambda_i)^{m_i}$ with $\lambda_i\in \oK$ is the minimal polynomial of $\phi_{\tau}$ on $\Hig^{\log}(\bL)_{\tau}$. Enlarging $K$, we may assume that all $\lambda_i$s lie in $K$. Consider the generalized eigenspace decomposition $\Hig^{\log}(\bL)_{\tau}=\bigoplus\limits_{i}\Hig^{\log}(\bL)_{\tau,\lambda_i}$ with respect to $\phi_{\tau}$. Since $\theta\circ\phi=(\phi-1)\circ\theta$, for a local section $s\in \Hig^{\log}(\bL)_{\tau,\lambda_i}(U)$ we have $$\phi_{\tau}(\theta(s))=\theta(\phi_{\tau}(s)+s)=\theta(\lambda_is+s)=(\lambda_i+1)\cdot\theta(s)$$ In other words, the Higgs field $\theta$ takes the summand $\Hig^{\log}(\bL)_{\tau,\lambda}$ to $\Hig^{\log}(\bL)_{\tau,\lambda+1}\otimes_{\cO_{\Kinf}}\Omega^1_{X_{\Kinf}/\Kinf}$, for every $\lambda$. 

If all $\lambda_i$s do not belong to a single coset of $\bZ$ in $\oK$ then $\Hig^{\log}(\bL)_{\tau}$ is decomposable into a direct sum of two non-zero Higgs bundles, which contradicts the fact that all endomorphisms are scalars. It remains to show that $\phi_{\tau}$ is semi-simple on $\Hig^{\log}(\bL)_{\tau}$. Define an endomorphism $f$ of $\Hig^{\log}(\bL)_{\tau}=\bigoplus\limits_i \Hig^{\log}(\bL)_{\tau,\lambda_i}$ by $\phi_{\tau}-\lambda_i$ on the summand $\Hig^{\log}(\bL)_{\tau,\lambda_i}$. By construction, it commutes with the Higgs field and is nilpotent. Therefore, it is equal to zero so $\prod\limits_i (\phi_{\tau}-\lambda_i)=0$ on $\Hig^{\log}(\bL)_{\tau}$. In other words, $\phi_{\tau}$ is semi-simple. We infer that for each $\tau$ there is an element $a_{\tau}\in\oK$ such that $\phi_{\tau}$ acts on $\Hig^{\log}(\bL)_{\tau}$ semi-simply with eigenvalues from $a_{\tau}+\bZ$. This means the the projective representation corresponding to $\bL$ is Hodge-Tate.
\end{proof}

\section{A criterion of de Rhamness for local systems with integral weights}

We stay in the same setup: $X$ is a smooth rigid-analytic variety over  $K$ embedded into a smooth proper $\oX$ such that $\oX\setminus X$ is a normal crossings divisor.

\begin{pr}\label{mock Ddr}
Suppose that $\bL$ is a $\bQ_p$-local system of rank $d$ on $X$ with Hodge-Tate weights in $\bZ$. Then the sheaf \begin{equation}\cA(\bL):=\RH^{\log}(\bL)^{\phi^d=0}\end{equation} is a locally free sheaf of rank $d$ of $\cO_{\oX_{\Kinf}}$-modules on $\oX_{\Kinf}$ equipped with an integrable connection and Griffiths transverse filtration, induced from those on $\RH^{\log}(\bL)$. Moreover, the Higgs bundle $(\bigoplus\limits_{i\in 
\bZ}\gr^i\cA(\bL),\gr(\nabla))$ is isomorphic to $\Hig^{\log}(\bL)$ (the isomorphism is not compatible with  the Sen endomorphisms).
\end{pr}

\begin{proof}
First, note that $\cA(\bL)$ is indeed a subsheaf of $\cO_{\oX_{\Kinf}}$-modules inside $\RH^{\log}(\bL)$ and it is preserved by the connection $\nabla$. Suppose that the weights of $\bL$ are contained in the set $\{i,i-1,\dots, i-n\}$ for $i\in\bZ,n\in\bN$. Then $\cA(\bL)$ is contained in $F^{-i}\RH^{\log}(\bL)$ because for $j<-i$ we have $(\gr^j\RH^{\log}(\bL))^{\phi^d=0}=(\Hig^{\log}(\bL)(j))^{\phi^d=0}=0$ as all eigenvalues of $\phi$ on $\Hig^{\log}(\bL)(j)$ are in the range $\{i+j,i+j-1,\dots i+j-n\}\subset \bZ_{<0}$. Similarly, the map $\cA(\bL)\to F^{[-i,-i+n]}\RH^{\log}(\bL):=F^{-i}\RH^{\log}(\bL)/F^{-i+n+1}\RH^{\log}(\bL)$ is injective.

The subquotient $F^{[-i,-i+n]}\RH^{\log}(\bL)$ is a locally free sheaf on $\oX_{\Kinf}$ (of rank $(n+1)\cdot d$) equipped with an $\cO_{\oX_{\Kinf}}$-linear endomorphism $\phi$ and $\cA(\bL)$ is its generalized eigenbundle with eigenvalue $0$. In particular, it is a direct summand of $F^{[-i,-i+n]}\RH^{\log}(\bL)$ and is therefore locally free.

 The generalized eigenspace of eigenvalue $0$ on the graded pieces $\gr^{j}\RH^{\log}(\bL)$ for $j\in[-i,-i+n]$ is given by $\Hig^{\log}(\bL)_{-j}$. Therefore we have an isomorphism of vector bundles \begin{equation}\label{grA bundle iso}\gr\cA(\bL)\simeq\bigoplus\limits_{j=-i}^{-i+n}\Hig^{\log}(\bL)_{-j}=\Hig^{\log}(\bL)\end{equation} The Higgs field $\gr(\nabla):\gr^j\cA(\bL)\to\gr^{j-1}\cA(\bL)\otimes\Omega^1_{X_{\Kinf}}$ is restricted from $\gr^j\RH^{\log}(\bL)\to\gr^{j-1}\RH^{\log}(\bL)\otimes\Omega^1_{X_{\Kinf}}$ and is thus given by $\theta_{\Hig^{\log}(\bL)}:\Hig^{\log}(\bL)_{-j}\to \Hig^{\log}(\bL)_{-j+1}\otimes\Omega^1_{X_{\Kinf}}$ so (\ref{grA bundle iso}) is compatible with Higgs fields.
\end{proof}

\begin{thm}\label{de rham main}
Suppose that $\bL$ is a $\bQ_p$-local system admitting a $\bZ_p$-lattice on $X$ with Hodge-Tate weights in $\bZ$ such that $\End_{X_{\oK}}(\bL)\otimes_{\bQ_p}C$ has no nilpotent elements. Then $\bL$ is de Rham.
\end{thm}

\begin{proof}
We are going to use the criterion for de Rhamness given by Lemma \ref{senoperator criteria}(i) which says that $\bL$ is de Rham if the nilpotent endomorphism $\phi$ of $\cA(\bL)$ is in fact zero. Suppose that it is not zero. Let $i$ be the maximal integer such that $\phi(F^m\cA(\bL))\subset F^{m+i}\cA(\bL)$ for all $m\in\bZ$. We will view $\phi$ as a morphism $\cA(\bL)\to \cA(\bL)(i)$ of filtered vector bundles with connection where the twist denotes the shift of the filtration: $F^m(\cA(\bL)(i))=F^{m+i}\cA(\bL)$. For at least one $m$ the induced map $\phi:F^m/F^{m+1}\to F^{m+i}/F^{m+i+1}$ is non-zero so $\phi$ induces a non-zero nilpotent endomorphism of the Higgs bundle $\Hig^{\log}(\bL)=\bigoplus\limits_{m\in\bZ}\gr^{m}\cA(\bL)$.

By Proposition \ref{main decompleted rh}(iii) the algebra $\End_{\oX_{\Kinf},\Higgs}\Hig^{\log}(\bL)\otimes_{\Kinf}C$ has no non-zero nilpotents under the assumption of the theorem. The same is then true for the algebra $\End_{\oX_{\Kinf},\Higgs}\Hig^{\log}(\bL)$. This gives us a contradiction.
\end{proof}

\begin{cor}\label{main th as cor}
Let $X$ be a smooth algebraic variety over $K$ and $\bL$ be an etale $\obQ_p$-etale local system such that $\End_{X_{\oK}}\bL=\obQ_p\cdot \Id$. Then there exists a character $\chi:G_K\to \obQ_p^{\times}$ such that $\chi\otimes\bL$ is de Rham.
\end{cor}

\begin{proof}
Using resolution of singularities, choose a smooth proper algebraic variety $\oX\supset X$ containing $X$ as a dense open such that $\oX\setminus X$ is a normal crossing divisor and apply the discussion of sections 2-5 to the analytifications of $X$, $\oX$ and $\bL$. We are in a position to do so because $\bL$, being a $\obQ_p$-local system on an algebraic variety, admits a $\obZ_p$-lattice. Theorem \ref{ht main} applied to $\bL$ provides the character $\chi$ and   Theorem \ref{de rham main} applied to a descent of $\bL\otimes\chi$ to a finite extension $E$ of $\bQ_p$ viewed as a $\bQ_p$-local system gives the result.
\end{proof}

\section{Assumptions in the relative Fontaine-Mazur conjecture}

For the duration of this section let $F$ be a number field and $\fX$ be a smooth scheme over $\cO_{F,S}$ where $S$ is a finite set of places of $F$. Denote by $X$ the generic fiber $\fX_F$. We show here that the first condition of the relative Fontaine-Mazur conjecture is automatic for geometrically irreducible local systems. Together with Corollary \ref{main th as cor} this completes the proof of Theorem \ref{fm intro}. I have learned about this fact from Daniel Litt. The proof, in the case of curves, appears in Step 2 of the proof of \cite[Theorem 1.1.3]{litt}. We include here (an essentially equivalent) proof for the sake of completeness. A very similar argument also appears in \cite[Proposition 4.1]{lz}.

\begin{pr}[D. Litt]\label{automatically unramified}
For any $\obQ_p$-local system $\bL$ on $X$ such that $\End_{X_{\oF}}\bL|_{X_{\oF}}=\obQ_p\cdot \Id$ there exists a finite set $S'\supset S$ such that $\bL$ extends to a $\obQ_p$-local system on $\fX_{\cO_{F,S'}}$.
\end{pr}

\begin{proof}
Any $\obQ_p$-local system descends to an $E$-local system for some finite extension $E$ of $\bQ_p$ and we can choose an $\cO_E$-lattice in this descent. We may therefore work with an $\cO_E$-local system $\bL$ such that $\End_{X_{\oF}}\bL|_{X_{\oF}}[\ip]=E\cdot \Id$. Note that it is enough to prove the claim for $X$ and $F$ replaced with $X_{F'}$ and $F'$ for some finite field extension $F'\supset F$. Hence, we may assume that $X$ admits an $F$-point $x:\Spec F\to X$ and we will pick a geometric point $\ox$ above $x$. Denote the representation corresponding to $\bL$ by $\rho:\pi_1^{\et}(X,\ox)\to GL_d(\cO_E)$.  Restriction $x^*\bL$ can be viewed as a Galois representation $\rho_{x}:G_F\to GL_d(\cO_E)$. Replace again $F$ by a finite extension to make the residual representation $\overline{\rho}_{x}:G_F\to GL_d(\cO_E/p^2)$ trivial.  Also, pick a smooth proper compactification $\oX\supset X$ such that $D=\oX\setminus X$ is a normal crossings divisor. After enlarging $S$ we may assume that $\oX$ has a smooth proper model $\ofX$ over $\cO_{F,S}$ such that $D$ extends to a horizontal normal crossings divisor $\fD\subset\ofX$ with $\fX=\ofX\setminus\fD$ and, moreover, $x$ extends to a point in $\fX(\cO_{F,S})$. 

Let now $S'$ be the union of $S$ with the set of places of $F$ such that (i) residue characteristic of $v$ divides the order of the finite group $GL_d(\cO_E/p^2)$ or (ii) the character $\det\rho_{x}$ of the Galois group $G_F$ is ramified at $v$. The set of places satisfying the second condition is finite by \cite[Corollary III-2.2]{serre} so $S'$ is finite. We will now prove that the local system $\bL$ extends to a local system on $\fX_{\cO_{F,S'}}$. 

First, we will show that the Galois representation $\rho_{x}$ is unramified at all places $v$ outside of $S'$. Let $v$ be any such place and $l$ be its residue characteristics. The prime $l$ does not divide the order of any of the groups $GL_d(\cO_E/p^n)$ for $n\in\bN$ by assumption so the representation $\rho: \pi_1^{\et}(X,\ox)\to GL_d(\cO_E)$ factors through the prime-to-$l$ quotient of $\pi_1^{\et}(X,\ox)$ which we will denote by $\pi_1^{\et}(X,\ox)^{(l')}$. The action of $G_F$ on $\pi_1(X_{\oF}, \ox)^{(l')}$ is unramified at $v$ by \cite[Corollary A.12]{lieblicholsson}, thanks to the fact that we assumed the existence of good compactification $\fX$ for $\fX$.  In other words, the image of the inertia subgroup $I_v\subset G_{F_v}$ inside $\pi_1^{\et}(X,\ox)^{(l')}$ commutes with the image of $\pi_1^{\et}(X_{\oF},\ox)^{(l')}\to \pi_1^{\et}(X,\ox)^{(l')}$. It means that $\rho_x(I_{F_v})$ commutes with every element of $\rho(\pi_1^{\et}(X_{\oF},\ox))$. By our assumption $\End_{X_{\oF}}\bL|_{X_{\oF}}[\ip]=E\cdot \Id$ every element of $GL_d(\cO_E)$ commuting with $\rho(\pi_1^{\et}(X_{\oF},\ox))$ has to be a scalar multiple of identity so for every $g\in I_{F_v}$ we have $\rho_x(g)=\lambda\cdot \Id$ for some $\lambda\in \cO_E^{\times}$. By our preparatory assumption $\lambda\equiv 1$ modulo $p^2$ and by the definition of $S'$ the power $\lambda^d$ is equal to $1$. This implies that $\lambda$ is in fact equal to $1$ so $\rho_x(I_{F_v})=\{\Id\}$ and $\rho_x$ is unramified at all places outside of $S'$. 

This now forces $\bL$ to extend to $\fX_{\cO_{F,S'}}$: let $F_n$ be the finite extension of $F$ trivializing $x^*(\bL/p^n)$, by the above it is unramified at all $v\not\in S'$ and to show that $\bL/p^n$ extends from $X$ to $\fX_{\cO_{F,S'}}$ it is enough to do so after the base change to $\cO_{F_n,S'}$. This extension exists by \cite[Lemma 57.31.7]{stacks}.
\end{proof}

Despite the fact that it is not clear whether for a geometrically irreducible $\bL$ on $X$ there always exists a character $\chi:G_F\to\obQ_p^{\times}$ of the global Galois group such that $\bL\otimes\chi$ is de Rham at all places above $p$, Moritz Kerz observed that one can still conclude from the relative Fontaine-Mazur conjecture that the underlying geometric local system comes from geometry.

\begin{lm}[M. Kerz]\label{square trick}
Assume that Conjecture \ref{fm original} holds for local systems on the variety $X\times X$. Then for any $\obQ_p$-local system $\bL$ on $X$ such that $\bL|_{X_{\oF}}$ is irreducible the local system $\bL|_{X_{\oF}}$ comes from geometry.
\end{lm}

\begin{proof}
Denote by $p_1,p_2:X\times X\to X$ the two projections and consider the local system $\bL\boxtimes\bL^{\vee}:=p_1^*\bL\otimes p_2^*\bL^{\vee}$. By Corollary \ref{main th as cor} the local system $(\bL\boxtimes\bL^{\vee})|_{X_{F_v}}$ is de Rham for every place $v$ above $p$ and $\bL\boxtimes\bL^{\vee}$ extends to a local system on $(\fX\times \fX)_{\cO_{F,S'}}$ for some finite set of places $S'$ by Proposition \ref{automatically unramified}. The assumptions of Conjecture \ref{fm original} are therefore satisfied for $\bL\boxtimes\bL^{\vee}$ so there exists an open $U\subset X$ such that $(\bL\boxtimes\bL^{\vee})|_U$ appears as a subquotient of the sheaf $R^i\pi_*\obQ_p(j)$ for some morphism $\pi:Y\to U$. 

Choose a point $x\in X(\oF)$ such that the intersection of $U$ with $\{x\}\times X$ is non-empty. Then the restriction $i_x^*(\bL\boxtimes \bL^{\vee})$ of the local system $(\bL\boxtimes\bL^{\vee})_{(X\times X)_{\oF}}$ along the inclusion $i_x:X_{\oF}=\{x\}\times X_{\oF}\to (X\times X)_{\oF}$ comes from geometry. However, this restriction is isomorphic to $\bL|_{X_{\oF}}\otimes \bL^{\vee}_x\simeq \bL_{X_{\oF}}^{\rk\bL}$ so the local system $\bL|_{X_{\oF}}$ on $X_{\oF}$ comes from geometry.
\end{proof}

\section{Rigidity for Hodge-Tate local systems on algebraic varieties}

For this section we go back into the setup of Section \ref{section2}: $X$ is a smooth geometrically connected rigid-analytic variety over $K$ with a smooth proper compactification $\oX\supset X$ such that $\oX\setminus X$ is a normal crossings divisor, the main example being provided by the analytification of a smooth algebraic variety.

\begin{lm}\label{vanishing rigidity}
If $\bL$ is a $\bQ_p$-local system on $X$ that admits a $\bZ_p$-lattice then any endomorphism $f:\Hig^{\log}(\bL)\to \Hig^{\log}(\bL)$ commuting with the Higgs field that vanishes at one closed point $x\in X_{\Kinf}$ has to be identically zero.
\end{lm}

\begin{proof}
By Theorem \ref{liu zhu rh} (v) and Proposition \ref{main decompleted rh} (ii) we have a commutative diagram

\[
\begin{tikzcd}
\End_{\oX_{\Kinf},\Higgs}\Hig^{\log}(\bL)\otimes_{\Kinf}C\arrow[d]\arrow[r,"\sim"] & \End_{X_{\oK}}\bL|_{X_{\oK}}\otimes_{\bQ_p}C\arrow[d]\\
\End_{\Kinf}\Hig^{\log}(\bL)_x\otimes_{\Kinf}C\arrow[r,"\sim"] & \End_{\bZ_p}\bL_{\ox}\otimes_{\bQ_p}C
\end{tikzcd}
\]

As $\bL\otimes_{\bZ_p}\bL^{\vee}$ is a local system the right vertical arrow is injective, so the left one is injective as well.

\end{proof}

\begin{pr}\label{ht locally free}
For any $\bQ_p$-local system on $X$ that admits a $\bZ_p$-lattice the associated coherent Higgs sheaf $D_{\HT}(\bL)$ is locally free.
\end{pr}

\begin{proof}
By \cite[Proposition 5.3]{kht} this statement amounts to the fact that $\bigoplus\limits_{i\in\bZ}\Hig^{\log}(\bL)^{\phi=i}|_{X_{\Kinf}}$ is locally free. Assume, by enlarging $K$, that all Hodge-Tate weights of $\bL$ belong to $K$. Let $\Hig^{\log}(\bL)=\bigoplus \Hig^{\log}(\bL)_{\lambda} $ be the generalized eigenbundle decomposition associated to the endomorphism $\phi$. Denote by $f$ the endomorphism of $\Hig^{\log}(\bL)$ defined by $\phi-\lambda$ on the component $\Hig^{\log}(\bL)_{\lambda}$. It commutes with the Higgs field because of the identity $\theta\circ\phi=(\phi-1)\circ \theta$. To prove that $\bigoplus\limits_{i\in\bZ}\Hig(\bL)^{\phi=i}|_{X_{\Kinf}}$ is locally free, it is enough to show that the rank of $f$ is constant across $X_{\Kinf}$. Equivalently, we need to prove that, if for some $m\in\bN$ the exterior power $\Lambda^m f\in \End_{\oX_{\Kinf},\Higgs}(\Lambda^m\Hig^{\log}(\bL))$ vanishes at a point $x\in X_{\Kinf}$, then $\Lambda^m f$ is identically zero. This follows from Lemma \ref{vanishing rigidity}. 
\end{proof}

\begin{cor}\label{ht rigidity main}
Suppose that $\bL$ is a $\bQ_p$-local system on $X$ admiting a $\bZ_p$-lattice. If for one classical point $x\in X$ the restriction $\bL_{\ox}$ is a Hodge-Tate representation of $\Gal(\overline{k(x)}/k(x))$ then $\bL$ is Hodge-Tate. In particular, for any other closed point $y\in X$ the Galois representation $\bL_{\oy}$ is Hodge-Tate as well.
\end{cor}

\begin{proof}
By the assumption, all Hodge-Tate weights of $\bL$ belong to $\bZ$. The proof of Proposition \ref{ht locally free} shows that for every $i\in\bZ$ the rank of $\phi-i$ is constant across $X$. Since the restriction of $\phi-i$ to $x$ annihilates all of the generalized eigenspace corresponding to the weight $i$, the operator $\phi-i$ does the same globally. Hence $\phi$ is semi-simple with integral eigenvalues.
\end{proof}

\section{Reducible local systems and Galois action on the fundamental group}

In this section, we extend Theorem \ref{main theorem} to local systems that are not necessarily geometrically irreducible by proving the following result. I learned about this generalization and its interpretation in terms of the Galois action on $\pi_1^{\et}(X_{\oK},\ox)$ from Alexander Beilinson. As before, $X$ is a smooth rigid-analytic variety that admits a smooth proper compactification $\oX\supset X$ over a complete discretely valued field $K$ of characteristic zero with a perfect residue field of characteristic $p$.

\begin{thm}\label{main reducible}
For any $\obQ_p$-local system $\bL$ on $X$ there exists a de Rham local system $\bM$ on $X$ such that there is an inclusion $\bL|_{X_{\oK}}\subset \bM|_{X_{\oK}}$ of local systems on $X_{\oK}$.
\end{thm}

\begin{proof}
Note that if the theorem is proven for the base change of $X$ to a finite extension $K'\supset K$ then it follows for the initial $X$ as well by considering the pushforward of $\bM$ along $X_{K'}\to X$. We may therefore assume that $X$ has a $K$-point $x$. Pick also a geometric point $\ox$ above $x$. We will view the stalk $\bL_{\ox}$ as a representation of $G_K$.

The local system $\bM$ can now be defined constructively. Denote by $\cF(\bL)$ the minimal sub-local system of $\bL^{\vee}\otimes_{\obQ_p}\bL_{\ox}$ whose stalk $\cF(\bL)_{\ox}$ contains the identity endomorphism $\id_{\bL_{\ox}}\in \bL^{\vee}_{\ox}\otimes\bL_{\ox}=\End_{\obQ_p}(\bL_{\ox})$. In terms of representations of $\pi_1^{\et}(X,\ox)$, the stalk $\cF(\bL)_{\ox}$ is the $\obQ_p$-span of the image of the geometric fundamental group $\pi_1^{\et}(X_{\oK},\ox)$ inside $\End_{\obQ_p}(\bL_{\ox})$. Note that $\cF(\bL)|_{X_{\oK}}$ is therefore minimal among local systems on $X_{\oK}$ contained in $(\bL^{\vee}\otimes\bL_{\ox})|_{X_{\oK}}$ and containing $\id_{\bL_{\ox}}$ in their stalk at $\ox$; this observation will play the main role in the proof of Proposition \ref{reducible explicit companion} below.

Now define $\bM=(\cF(\bL)^{\vee})^{\oplus\rk\bL}$. By construction, $\cF(\bL)\otimes\bL_x^{\vee}$ surjects onto $\bL^{\vee}$, therefore $\bL|_{X_{\oK}}$ injects into $(\cF(\bL)^{\vee}\otimes\bL_x)|_{X_{\oK}}=\bM|_{X_{\oK}}$. The proof of the theorem now amounts to the following proposition.
\end{proof}

\begin{pr}\label{reducible explicit companion}
For any $\obQ_p$-local system $\bL$, the local system $\cF(\bL)$ defined above is de Rham.
\end{pr}

\begin{proof}
Note that the subspace $\cF(\bL)_{\ox}\subset \End_{\obQ_p}(\bL_{\ox})$ only depends on $\bL|_{X_{\oK}}$ so, in particular, for a finite extension $K'\supset K$ we have $\cF(\bL|_{X_{K'}})=\cF(\bL)|_{X_{K'}}$. We may thus freely replace $K$ by a finite extension. If $\bK$ is a local system of $E$-vector spaces on $X$ for a finite extension $E\supset \bQ_p$ we will denote by $\cF(\bK)$ the sub local system of $\bK^{\vee}\otimes_E\bK_{\ox}$ defined just as above in the case of $\obQ_p$-local systems. If $\bK_{\bQ_p}$ denotes the same $\bK$, considered as a $\bQ_p$-local system of rank $\rk_{E}\bK\cdot[E:\bQ_p]$, then $\cF(\bK)$ can be recovered from $\bK_{\bQ_p}$ by the formula $\cF(\bK)=\cF(\bK_{\bQ_p})\otimes_{E\otimes_{\bQ_p}E,m}E$. To prove the proposition, we may therefore work with $\bL$ a $\bQ_p$-local system to begin with. 

As in the proof of Theorem \ref{ht main}, let us start by translating the defining property of $\cF(\bL)$ into a statement about the corresponding Higgs bundle $\Hig^{\log}(\cF(\bL))$. By definition, $\cF(\bL)|_{X_{\oK}}$ has no proper sub local systems $\bM'\subset \cF(\bL)|_{X_{\oK}}$ whose stalk ${\bM}'_{\ox}$ contains $\id_{\bL_{\ox}}\in \cF(\bL)_{\ox}$. In particular, any endomorphism $f\in \End_{X_{\oK}}(\cF(\bL)|_{X_{\oK}})$ satisfying $f_{\ox}(\id_{\bL_{\ox}})=0$ has to be identically zero, as otherwise $\ker f$ would give an example of such $\bM'$. Under the comparison isomorphism $\cF(\bL)_{\ox}\otimes_{\bQ_p}C\simeq \Hig^{\log}(\cF(\bL))_x\otimes_{\Kinf}C$ the element $\id_{\bL_{\ox}}\otimes 1$ of the left hand side gets carried to the element $\id_{\Hig^{\log}(\bL)_x}\otimes 1\in \Hig^{\log}(\cF(\bL))_x\otimes_{\Kinf}C\subset \End_{\Kinf}(\Hig^{\log}(\bL)_x)\otimes_{\Kinf}C$ because $\id_{\bL_{\ox}}\in \End(\bL_{\ox})$ is stable under the Galois action. By comparison isomorphism of Proposition \ref{main decompleted rh} (iii), this implies that any endomorphism $f\in \End_{\oX_{\Kinf},\Higgs}\Hig^{\log}(\cF(\bL))$ satisfying $f(\id_{\Hig^{\log}(\bL)_x})=0$ has to be zero.

We can now conclude that $\cF(\bL)$ has integral Hodge-Tate weights. The direct sum $\bigoplus\limits_{i\in\bZ}\Hig^{\log}(\cF(\bL))_i\subset \Hig^{\log}(\cF(\bL))$ of generalized eigenbundles of $\phi$ with integral eigenvalues is a direct summand of the Higgs bundle $\Hig^{\log}(\cF(\bL))$ or, in other words, there is an idempotent $e\in\End_{X_{\Kinf},\Higgs}\Hig^{\log}(\cF(\bL))$ such that $\ker(e)=\bigoplus\limits_{i\in\bZ}\Hig^{\log}(\cF(\bL))_i$. Since $\id_{\Hig^{\log}(\bL)_x}\in \End(\Hig^{\log}(\bL)_x)=\End(\Hig(\bL_{\ox}))$ spans the result of applying $\Hig$ to the trivial Galois representation, $\phi(\id_{\Hig^{\log}(\bL)_x})=0$ and hence $e(\id_{\Hig^{\log}(\bL)_x})=0$. This forces $e$ to be zero, so every eigenvalue of $\phi$ on $\Hig^{\log}(\cF(\bL))$ is an integer.

We can now finish the proof of the proposition analogously to the proof of Theorem \ref{de rham main}. Consider the sheaf $\cA(\cF(\bL))=\RH^{\log}(\cF(\bL))^{\phi^{\rk(\cF(\bL))}=0}$ as in Proposition \ref{mock Ddr}. If $\cF(\bL)$ is not de Rham, then the Sen endomorphism on $\cA(\cF(\bL))$ induces, as described in the proof of Theorem \ref{de rham main}, a non-zero nilpotent endomorphism of the Higgs bundle $\Hig^{\log}(\cF(\bL))$. This endomorphism, however, annihilates the element $\id_{\Hig^{\log}(\bL)_x}\in \Hig^{\log}(\cF(\bL))_x$, so it has to be zero by the discussion above. Hence the local system $\cF(\bL)$ is de Rham.
\end{proof}

It is unclear if it is always possible to choose $\bM$ such that $\bL|_{X_{\oK}}\simeq \bM|_{X_{\oK}}$ (note that this is the case for a geometrically irreducible $\bL$ by Theorem \ref{main theorem}). In particular, I do not know the answer to the following question

\begin{question}
If $\bL_1,\bL_2$ are de Rham local systems on $X$ and $0\to\bL_1\to\bL\to\bL_2\to 0$ is an extension of local systems on $X$, does there exists an extension $0\to\bL_1\to\bL'\to\bL_2\to 0$ such that $\bL'$ is de Rham and there is an isomorphism $\bL|_{X_{\oK}}\simeq \bL'|_{X_{\oK}}$ of extensions of $\bL_2|_{X_{\oK}}$ by $\bL_1|_{X_{\oK}}$?
\end{question}

From now on we assume that $x$ is either a $K$-point of $X$ or a tangential base point of $X$ in the sense of \cite[\S 15]{delignetrois} and $\ox$ is a geometric base point above $x$. We can reformulate the above result in terms of the Galois action on the pro-algebraic completion of the fundamental group $\pi_1^{\et}(X_{\oK},\ox)$. We will now recall all the necessary information about pro-algebraic completions of pro-finite groups and refer the reader to Section 3.1 of \cite{drinfeld} for a more thorough introduction.

\begin{definition}
For a compact Hausdorff topological group $\Gamma$ the {\it pro-algebraic completion} $\Gamma^{\pro-\alg}$ is the pro-algebraic group over $\obQ_p$, equipped with a continuous map $\alpha:\Gamma\to \Gamma^{\pro-\alg}(\obQ_p)$ defined by the following universal property. For any continous homomorphism $\rho:\Gamma\to GL_n(\obQ_p)$ there exists a unique algebraic morphism $\rho^{\alg}:\Gamma^{\pro-\alg}\to GL_{n,\obQ_p}$ such that $\alpha\circ\rho^{\alg}=\rho$.
\end{definition}

Here, for a pro-algebraic group $G$ expressed as a limit $\lim\limits_{i\in I}G_i$ of algebraic groups $G_i$ we endow $G(\obQ_p)=\lim\limits_{i\in I}G_i(\obQ_p)$ with a structure of a topological group given by the limit of the topological groups $G_i(\obQ_p)$, each carrying the standard $p$-adic topology. Note that, as $\Gamma$ is assumed to be compact and Hausdorff, any continuous homomorphism $\Gamma\to GL_n(\obQ_p)$ factors through $GL_n(E)$ for some finite extension of $E$, by a Baire category theorem argument, cf. \cite[Lemma 2.2.1.1]{breuilmezard}.

We will denote by $\obQ_p[\Gamma^{\pro-\alg}]$ the ring of regular functions on $\Gamma^{\pro-\alg}$. Our main objects of interest are pro-algebraic compeltions of the profinite groups $\pi_1^{\et}(X_{\oK},\ox)$ and $\pi_1^{\et}(X,\ox)$, which we denote by $\pi_1^{\pro-\alg}(X_{\oK},\ox)$ and $\pi_1^{\pro-\alg}(X,\ox)$, respectively. The continuous action of $G_K$ on $\pi_1^{\et}(X_{\oK},\ox)$ induces an action on $\obQ_p[\pi_1^{\pro-\alg}(X_{\oK},\ox)]$.

\begin{pr}\label{galois on proalg pi1}
Any finite-dimensional $\obQ_p$-representation $V$ of $G_K$ that can be embedded into $\obQ_p[\pi_1^{\pro-\alg}(X_{\oK},\ox)]$ is de Rham.
\end{pr}

\begin{proof}
The span of the $G_K$-orbit of a function $f\in \obQ_p[\pi_1^{\pro-\alg}(X_{\oK},\ox)]$ is finite-dimensional if and only if $f$ extends to a regular function on the pro-algebraic completion of the arithmetic fundamental group $\pi_1^{\et}(X,\ox)$. Therefore $V$ is contained in the image of the restriction map $\obQ_p[\pi_1^{\pro-\alg}(X,\ox)]\to \obQ_p[\pi_1^{\pro-\alg}(X_{\oK},\ox)]$.

By Theorem \ref{main reducible}, for any continuous homomorphism $\rho:\pi_1^{\et}(X,\ox)\to GL_n(\obQ_p)$ there exists another one $\rho':\pi_1^{\et}(X,\ox)\to GL_{n+m}(\obQ_p)$ such that the image of the restriction map $\rho^*:\obQ_p[GL_{n,\obQ_p}]\to \obQ_p[\pi_1^{\pro-\alg}(X_{\oK,\ox})]$ is contained in the image of the restriction map $\rho'^*$, and $\rho'$ corresponds to a de Rham local system. 

Denote by $W$ the Galois representation obtained by restricting $\rho'$ along $x_*:G_K=\pi_1^{\et}(\Spec K)\to\pi_1^{\et}(X, \ox)$. If $x\in X(K)$ is a classical base point then $W$ is evidently de Rham. If $x$ is a tangential base point, the fact that $W$ is de Rham was proven in \cite[Corollary 4.3.4]{dllz} as a consequence of their results on the compatibility of the functor $D_{\dR}$ with the functor of nearby cycles. The image $\rho'^*(\obQ_p[GL_{n+m,\obQ_p}])\subset \obQ_p[\pi_1^{\pro-\alg}(X_{\oK},\ox)]$ is a direct sum of Galois representations of the form $W^{\otimes a}\otimes W^{\vee\otimes b}\otimes (\det W)^{c}$ for varying $a,b\in\bZ_{\geq 0},c\in\bZ$. Since $W$ is de Rham, the image of $\rho'^*$ is thus a union of de Rham representations. We have therefore proven that the image of $\obQ_p[\pi_1^{\pro-\alg}(X,\ox)]\to \obQ_p[\pi_1^{\pro-\alg}(X_{\oK},\ox)]$ is a union of finite dimensional de Rham representations, because any function on the pro-algebraic completion $\pi_1^{\pro-\alg}(X,\ox)$ factors through some morphism of pro-algebraic groups $\pi_1^{\pro-\alg}(X,\ox)\to GL_{n,\obQ_p}$. 
\end{proof}

\begin{cor}
Let $X$ be a smooth variety over a number field $F$, equipped with a base point $x$ that is either a classical point $x\in X(F)$ or a tangential base point. Any finite-dimensional representation of $G_F$ contained in $\obQ_p[\pi_1^{\pro-\alg}(X_{\oF},\ox)]$ is geometric in the sense of Fontaine-Mazur: it is de Rham at all places $v$ of $F$ above $p$ and it is almost everywhere unramified.
\end{cor}

\begin{proof}
Let $V\subset \obQ_p[\pi_1^{\pro-\alg}(X_{\oF},\ox)]$ be a finite-dimensional subspace stable under the action of $G_F$. The de Rhamness has just been proven, so it remains to prove that $V$ is an almost unramified representation. By definition, there exists a map of pro-algebraic groups $f:\pi_1^{\pro-\alg}(X_{\oF},\ox)\to GL_{n,\obQ_p}$ such that every function from $V$ factors through it. Since the image of $\pi_1^{\et}(X_{\oF},\ox)$ under $f$ is contained in $GL_n(\cO_E)\subset GL_n(\obQ_p)$ for some finite extension $E\supset \bQ_p$, there exists a finite set $S$ of rational primes such the restriction of $f$ to $\pi_1^{\et}(X_{\oF},\ox)$ factors through the pro-$S$ completion $\pi_1^{\et}(X_{\oF},\ox)^{(S)}$ of this group. Therefore, $V$ lies in the image of the canonical map $\obQ_p[\pi_1^{\pro-\alg}(X_{\oF},\ox)^{(S)}]\to \obQ_p[\pi_1^{\pro-\alg}(X_{\oF},\ox)]$. After enlarging $S$ we may assume that $X$ is the generic fiber of a smooth $\cO_{F,S}$-scheme $\fX$ that can be established as the complement to a horizontal normal crossings divisor $\fD\subset \overline{\fX}$ in a smooth proper $\overline{\fX}$ and, moreover, $x$ extends to a point in $\fX(\cO_{F,S})$. By Lemma \ref{pi1 specialization} below the action of $G_F$ on $\pi_1^{\pro-\alg}(X_{\oF},\ox)^{(S)}$ is unramified at all places not lying above $S$. Therefore $V$ is almost everywhere unramified.
\end{proof}

\begin{lm}\label{pi1 specialization}
Let $v$ be a finite place of $F$ of residue characteristics $l=\mathrm{char}\, k(v)$. Suppose that $\cY$ is a smooth scheme over $\cO_{F_v}$ that can be embedded into a smooth proper $\cO_{F_v}$ scheme $\overline{\cY}$ such that $\overline{\cY}\setminus\cY$ is a relative normal crossings divisor. If $y$ is a classical or tangential base point of $\cY$ then there is an isomorphism

\begin{equation}
\pi_1^{\et}(\cY_{\oF_v},\oy_{F_v})^{(l')}\simeq \pi_1^{\et}(\cY_{\overline{k(v)}},\oy_{k(v)})^{(l')}
\end{equation}

compatible with the action of $G_{F_v}$ where the action on the right hand side is given by precomposing the natural action of $G_{k(v)}$ with the surjection $G_{F_v}\to G_{k(v)}$. Here $y_{F_v}$ (resp. $y_{k(v)}$) denotes the base point of $\cY_{F_v}$ (resp. $\cY_{k(v)}$) induces by $y$.
\end{lm}

\begin{proof}
Let $\cO^{nr}_{F_v}$ denote the completed ring of integers of the maximal unramified extension of $F_v$. By \cite[Theorem A.7]{lieblicholsson} (or \cite[Corollaire X.3.9]{SGA1} if $\cY=\overline{\cY}$) the categories of Galois covers of degrees prime to $l$ of $\cY_{\oF_v}$, $\cY_{\cO^{nr}_{F_v}}$ and $\cY_{\overline{k(v)}}$ are equivalent via the natural functors. In particular, these equivalences are compatible with the fiber functors induced by $y$ which implies the result.
\end{proof}

\bibliographystyle{alpha}
\bibliography{bibderham}

\end{document}